\documentclass[reqno,11pt]{amsart}

\usepackage{amsmath,amssymb,amsfonts,amsthm,amscd,amsxtra,mathrsfs,indentfirst}
\usepackage[top=3cm,bottom=3.5cm,right=2.5cm,left=2.2cm]{geometry}
\usepackage{mathtools,etoolbox}
\usepackage{esint,slashed,graphicx,caption,xcolor,booktabs}
\usepackage{enumerate}
\usepackage[hyphens]{url}
\usepackage{hyperref}
\usepackage{bookmark}

\allowdisplaybreaks
\mathtoolsset{showonlyrefs=true}
\DeclareRobustCommand{\textsection}{\ifmmode\mathsection\else\S\fi}

\usepackage[
    backend=biber,
    style=numeric,
    sorting=nyt,
    maxbibnames=99,
    giveninits=true,
]{biblatex}
\renewbibmacro{in:}{}
\DeclareFieldFormat[article,inproceedings,incollection,unpublished]{title}{#1}
\DeclareFieldFormat[article,inproceedings,incollection,unpublished]{citetitle}{#1}
\DeclareSourcemap{
    \maps[datatype=bibtex]{
        \map{\step[fieldsource=shortjournal,fieldtarget=journaltitle]}
    }
}

\newtheorem{theorem}{Theorem}[section]
\newtheorem{lemma}[theorem]{Lemma}
\newtheorem{proposition}[theorem]{Proposition}

\newtheorem{corollary}[theorem]{Corollary}
\newtheorem{remark}[theorem]{Remark}
\newtheorem{example}[theorem]{Example}

\newcommand{\abs}[1]{\lvert#1\rvert}
\newcommand{\Abs}[1]{\left\lvert#1\right\rvert}

\newcommand{\norm}[1]{\lVert#1\rVert}

\newcommand{\rd}{{\rm d}}
\newcommand{\rdA}{\rd A}

\newcommand{\rD}{{\rm D}}

\newcommand{\rtr}{{\rm tr}}
\newcommand{\rdiv}{{\rm div}}

\newcommand{\p}{\partial}
\def\<{\langle}
\def\>{\rangle}
\def\S{{\mathbb S}}
\def\R{\mathbb{R}}
\def\C{{\mathbb C}}

\newcommand{\Sph}{\S^{N-1}}
\newcommand{\E}{\mathcal{E}}
\newcommand{\Ein}{\mathcal{E}_{\mathrm{in}}}
\newcommand{\G}{\mathcal{G}}
\newcommand{\Cc}{\mathcal{C}}

\newcommand{\eq}[1]{\begin{equation}\allowdisplaybreaks\begin{alignedat}{2} #1 \end{alignedat}\end{equation}}
\numberwithin{equation}{section}

\begin{document}

\title[An $L^2$ Ahlfors--Beurling inequality in every dimension]
{An $L^2$ Ahlfors--Beurling inequality in every dimension}

\author{Guofang Wang}
\address{Albert-Ludwigs-Universit\"at Freiburg, Mathematisches Institut,\newline
Ernst-Zermelo-Str. 1, D-79104 Freiburg, Germany}
\email{guofang.wang@math.uni-freiburg.de}

\author{Mingwei Zhang}
\address{Wuhan University, School of Mathematics and Statistics,
430072 Wuhan, China and Albert-Ludwigs-Universit\"at Freiburg, Mathematisches Institut,\newline
Ernst-Zermelo-Str. 1, D-79104 Freiburg, Germany}
\email{zhangmwmath@whu.edu.cn}

\date{October 2026}

\begin{abstract}
Let $\Omega\subset\R^N$ ($N\geq2$) be measurable with $0<\abs{\Omega}<\infty$, and let $u_\Omega$ be its logarithmic potential ($N=2$) or Newtonian potential ($N\ge3$), normalized so that $\Delta u_\Omega=\chi_\Omega$. We prove the sharp inequality
\eq{
    \int_\Omega\abs{\nabla u_\Omega}^2\,\rd x
    \leq\frac{1}{2N\omega_{N-1}}
    \iint_{\Omega\times\Omega}\frac{\rd x\,\rd y}{\abs{x-y}^{N-2}},
}
where $\omega_{N-1}=|\S^{N-1}|$, with equality if and only if $\Omega$ is a ball up to a null set. In dimension two this settles the $L^2$ Ahlfors--Beurling conjecture of Fleeman and Khavinson. For $N\ge3$ it is equivalent to the scale-free bound that at most the fraction $(N-2)/(2N)$ of the total field energy lies inside $\Omega$. We also obtain an explicit deficit identity, a sharp extension to bounded nonnegative densities, and consequences for torsional rigidity, Bergman analytic content, Brownian motion, and aggregation equations.
\end{abstract}

\maketitle
\tableofcontents

\section{Introduction}\label{sec:results}

Let $\Omega\subset\C$ be measurable with $0<\abs{\Omega}<\infty$. Its Cauchy transform is
\eq{
    \Cc_\Omega(\zeta)=\frac{1}{\pi}\int_\Omega\frac{\rdA(z)}{\zeta-z},
    \qquad\zeta\in\C.
}
The Ahlfors--Beurling inequality \cite{AhlforsBeurling50} asserts that
\eq{\label{eq:AB}
    \abs{\Cc_\Omega(\zeta)}\leq\sqrt{\frac{\abs{\Omega}}{\pi}}\qquad(\zeta\in\C),
}
with equality if and only if $\Omega$ agrees a.e.\ with a disc whose boundary contains $\zeta$. This yields Alexander's sharp upper bound for analytic content \cite{Al} and, through it, a proof of the isoperimetric inequality \cite{GK,BK}.

Fleeman and Khavinson \cite{FK} proposed the following mean-square strengthening:
\eq{\label{eq:FK}
    \norm{\Cc_\Omega}_{L^2(\Omega)}\leq\frac{\abs{\Omega}}{\sqrt{2\pi}}.
}
The inequality is sharp, with equality for discs. By contrast, integrating \eqref{eq:AB} over $\Omega$ gives only
\eq{
    \int_\Omega\abs{\Cc_\Omega}^2\,\rdA\leq\frac{\abs{\Omega}^2}{\pi},
}
which is off by a factor $2$. As observed in \cite{FK}, \eqref{eq:FK} implies the sharp Bergman analytic content bound and the Saint-Venant inequality; the Bergman bound itself was obtained in \cite{FK} from the sharp Hankel estimate of Olsen and Reguera \cite{OR}.

Estimating $\norm{\Cc_\Omega}_{L^2(\Omega)}$ via the $L^2(\Omega)\to L^2(\Omega)$ operator norm of the interior Cauchy transform is not sharp either. Even for a disc of radius $r$, that operator norm equals $2r/j_0$, where $j_0$ is the first positive zero of $J_0$ \cite{AKL}, leading to a constant larger than that in \eqref{eq:FK} by the factor $2\sqrt2/j_0>1$. Moreover, the identity $\norm{\Cc}_{L^2\to L^2}=2/\sqrt{\Lambda_1(\Omega)}$ claimed in \cite{Dos} fails in general; Kalaj \cite{Ka} proved strict inequality for bounded simply connected domains with smooth boundary other than discs, and for annuli. Here $\Lambda_1(\Omega)$ denotes the first Dirichlet eigenvalue. The present paper instead estimates the single function $\Cc_\Omega=\Cc\chi_\Omega$.

In this paper, we prove \eqref{eq:FK}, characterize the equality case, and extend the estimate to every dimension. Throughout, let $N\geq2$, and set
\eq{
    \omega_{N-1}=|\S^{N-1}|=\frac{2\pi^{N/2}}{\Gamma(N/2)},
    \qquad s_N=\frac{\omega_{N-1}}{N}.
}
Thus $s_N$ is the volume of the unit ball in $\R^N$. Let $\Omega\subset\R^N$ be measurable with $0<\abs{\Omega}<\infty$. Define
\eq{
    u_\Omega(x)\coloneqq \frac{1}{2\pi}\int_\Omega
    \log\frac{\abs{x-y}}{1+\abs{y}}\,\rd y \qquad \text{for}\quad N=2,
}
and
\eq{
    u_\Omega(x)\coloneqq -\frac{1}{(N-2)\omega_{N-1}}\int_\Omega
    \frac{\rd y}{\abs{x-y}^{N-2}} \qquad \text{for}\quad N\geq3.
}
The planar normalization makes the integral converge for every set of finite area; if $\int_\Omega\log(1+\abs{y})\,\rd y<\infty$, it differs from the usual logarithmic potential by an additive constant. In every dimension,
\eq{\label{eq:grad}
    \nabla u_\Omega(x)=\frac{1}{\omega_{N-1}}\int_\Omega
    \frac{x-y}{\abs{x-y}^N}\,\rd y,
    \qquad\Delta u_\Omega=\chi_\Omega.
}
See Lemma~\ref{lem:pot} for the regularity and decay properties of $u_\Omega$. For $N=2$, one has $\abs{\Cc_\Omega}=2\abs{\nabla u_\Omega}$, so \eqref{eq:FK} is equivalent to
\eq{\label{eq:FK2}
    \int_\Omega\abs{\nabla u_\Omega}^2\,\rd x\leq\frac{\abs{\Omega}^2}{8\pi}.
}
Set
\eq{
    \Ein(\Omega)=\int_\Omega\abs{\nabla u_\Omega}^2\,\rd x,
    \qquad I(\Omega)=\iint_{\Omega\times\Omega}
    \frac{\rd x\,\rd y}{\abs{x-y}^{N-2}}.
}
Up to physical constants, $\nabla u_\Omega$ is the electrostatic field of a body of unit charge density, and $\Ein(\Omega)$ is its interior field energy.
For $N=2$, $I(\Omega)=\abs{\Omega}^2$. For $N\geq3$, the total field energy is
\eq{
    \E(\Omega)=\int_{\R^N}\abs{\nabla u_\Omega}^2\,\rd x
    =\frac{I(\Omega)}{(N-2)\omega_{N-1}}.
}
For a ball $B$ centred at $c$, one has $\nabla u_B(x)=(x-c)/N$ for $x\in B$.

\begin{theorem}\label{thm:main}
For every measurable $\Omega\subset\R^N$ with $0<\abs{\Omega}<\infty$, we have
\eq{\label{eq:main}
    \int_\Omega\abs{\nabla u_\Omega}^2\,\rd x
    \leq\frac{1}{2N\omega_{N-1}}\iint_{\Omega\times\Omega}
    \frac{\rd x\,\rd y}{\abs{x-y}^{N-2}},
}
with equality if and only if $\Omega$ is a ball up to a null set. In particular,
\begin{enumerate}[(i)]
\item if $N=2$, then $\displaystyle
\int_\Omega\abs{\Cc_\Omega}^2\,\rdA\leq\frac{\abs{\Omega}^2}{2\pi}$;
\item if $N\geq3$, then $\displaystyle
\Ein(\Omega)\leq\frac{N-2}{2N}\E(\Omega)$.
\end{enumerate}
\end{theorem}

For $N\geq3$, Theorem~\ref{thm:main} says that at least the fraction $\frac{N+2}{2N}$ of the total field energy lies outside $\Omega$; in $\R^3$ the interior fraction is at most $1/6$.
In the plane, if $\chi_\Omega$ is the vorticity of an incompressible flow, then $v=\nabla^\perp u_\Omega$ and \eqref{eq:FK2} bounds the kinetic energy stored inside the patch. The sharp maximizer is the disc. For an ellipse with semiaxes $a,b$, the interior energy equals $\abs{\Omega}^2\,(4ab/(a+b)^2)/(8\pi)$ (Example~\ref{ex:ellipsoids}). The total planar kinetic energy is infinite; the finite logarithmic interaction energy considered in Section~\ref{sec:apps} is a different functional.

Our proof represents $\nabla u_\Omega$ in terms of ray lengths and combines Bessel's inequality on the sphere with an integral-geometric averaging over lines. This yields the following deficit identity.
For $x\in\R^N$ and $\theta\in\Sph$, set
\eq{
    \lambda(x,\theta)=\abs{\{r>0:x+r\theta\in\Omega\}},
    \qquad\delta(x,\theta)=\lambda(x,\theta)-\lambda(x,-\theta).
}
For each fixed $x$, both $\lambda(x,\theta)$ and $\delta(x,\theta)$ are finite for a.e.\ $\theta$. Moreover, $\delta(x,\cdot)\in L^2(\Sph)$ for a.e.\ $x\in\Omega$ (see Lemma~\ref{lem:rays} and Section~\ref{sec:proof}). Let $P_1$ be the orthogonal projection of $L^2(\Sph)$ onto $\operatorname{span}\{\theta_1,\ldots,\theta_N\}$, and let $\norm{\cdot}$ denote the $L^2(\Sph)$ norm. For $p\in\theta^\perp$, set
\eq{
    E_{\theta,p}=\{t\in\R:p+t\theta\in\Omega\},
    \qquad\sigma(\theta,p)=\abs{E_{\theta,p}}.
}
For measurable $E\subset\R$ of finite length $\sigma$, define
\eq{\label{G}
    \G(E)=\iint_{E\times E}\abs{s-t}\,\rd s\,\rd t-\frac{\sigma^3}{3}.
}
By Lemma~\ref{lem:1d}, one has $\G(E)\geq0$, with equality if and only if $E$ agrees a.e.\ with an interval.

\begin{theorem}\label{thm:deficit}
Under the hypotheses of Theorem~\ref{thm:main},
\eq{\label{eq:id1}
    \Ein(\Omega)=\frac{1}{12N\omega_{N-1}}\int_{\Sph}\int_{\theta^\perp}
    \sigma(\theta,p)^3\,\rd p\,\rd\theta-\frac{1}{4N\omega_{N-1}}\int_\Omega
    \norm{\delta(x,\cdot)-P_1\delta(x,\cdot)}^2\,\rd x,
}
and
\eq{\label{eq:id2}
    \frac{I(\Omega)}{2N\omega_{N-1}}=\frac{1}{12N\omega_{N-1}}\int_{\Sph}\int_{\theta^\perp}
    \sigma(\theta,p)^3\,\rd p\,\rd\theta+\frac{1}{4N\omega_{N-1}}\int_{\Sph}\int_{\theta^\perp}
    \G(E_{\theta,p})\,\rd p\,\rd\theta.
}
Consequently,
\eq{\label{eq:deficit}
    \mathcal{D}(\Omega)&\coloneqq \frac{I(\Omega)}{2N\omega_{N-1}}-\Ein(\Omega)\\
    &=\frac{1}{4N\omega_{N-1}}\left[
    \int_\Omega\norm{\delta(x,\cdot)-P_1\delta(x,\cdot)}^2\,\rd x
    +\int_{\Sph}\int_{\theta^\perp}\G(E_{\theta,p})\,\rd p\,\rd\theta
    \right].
}
\end{theorem}

Both terms in \eqref{eq:deficit} are nonnegative. The second term vanishes precisely when almost every line section agrees a.e.\ with an interval; in particular, this forces $\Omega$ to agree a.e.\ with a convex set. The first term vanishes when the ray asymmetry $\delta(x,\cdot)$ is linear in $\theta$ for a.e.\ $x\in\Omega$.
For a ball centred at $c$, one has $\delta(x,\theta)=-2\<x-c,\theta\>$ and both terms vanish. Conversely, equality forces $D^2u_\Omega=I/N$ in the convex representative; a maximum-principle argument for angular derivatives then yields a ball (see Section~\ref{sec:rigidity}). The common term in \eqref{eq:id1}--\eqref{eq:id2} yields the intermediate estimate
\eq{
    \Ein(\Omega)\leq\frac{1}{12N\omega_{N-1}}\int_{\Sph}\int_{\theta^\perp}
    \sigma(\theta,p)^3\,\rd p\,\rd\theta\leq\frac{I(\Omega)}{2N\omega_{N-1}}.
}
For convex sets the second inequality is an equality (the chord power formula \cite{San}); in the plane this reduces to Crofton's identity $\int_{\S^1}\int_\R\sigma^3\,\rd p\,\rd\theta=6\abs{\Omega}^2$.

\begin{remark}\label{rem:deficitforms}
The nonconvexity term also has the form
\eq{\label{eq:pair}
    \frac{1}{4N\omega_{N-1}}\int_{\Sph}\int_{\theta^\perp}\G(E_{\theta,p})\,\rd p\,\rd\theta
    =\frac{1}{2N\omega_{N-1}}\iint_{\Omega\times\Omega}
    \frac{\mathcal{H}^1([x,y]\backslash\Omega)}{\abs{x-y}^{N-1}}\,\rd x\,\rd y.
}
This follows from Lemma~\ref{lem:BP} and the expression for $\G$ in
Lemma~\ref{lem:1d}(c). Since $\delta(x,-\theta)=-\delta(x,\theta)$, its spherical
harmonic expansion contains only odd degrees. If $P_\ell$ denotes projection
onto harmonics of degree $\ell$, then
\eq{
    \norm{\delta(x,\cdot)-P_1\delta(x,\cdot)}^2
    =\sum_{\substack{\ell\geq3\\\ell\ \mathrm{odd}}}
    \norm{P_\ell\delta(x,\cdot)}^2.
}
For an ellipsoid the pair term vanishes, so only the spherical harmonics of odd degree at least three contribute; see Example~\ref{ex:ellipsoids}.
\end{remark}

\begin{remark}\label{sec:norearr}\label{rem:scope}
The interior energy is trilinear in $\chi_\Omega$:
\eq{
    \Ein(\Omega)=\frac{1}{\omega_{N-1}^2}\iiint_{\Omega\times\Omega\times\Omega}
    \frac{\<x-y,x-z\>}{\abs{x-y}^N\abs{x-z}^N}\,\rd x\,\rd y\,\rd z.
}
Its kernel changes sign, so the Riesz--Sobolev and Brascamp--Lieb--Luttinger
inequalities \cite{Ri,So,BLL} do not directly give \eqref{eq:main}. Our proof uses no symmetrization. Instead, the polar Jacobian cancels the factor $\abs{x-y}^{1-N}$, and the line Jacobian cancels $\abs{x-y}^{2-N}$ (see Lemmas~\ref{lem:rays} and~\ref{lem:BP}). These cancellations are specific to the logarithmic and Newtonian potentials.
\end{remark}

As a consequence, we obtain the sharp bound for the interior energy at fixed volume.
Let $\Omega^*$ be the ball centered at $0$ with $\abs{\Omega^*}=\abs{\Omega}$.
For $N\geq3$, the Riesz--Sobolev inequality gives $I(\Omega)\leq I(\Omega^*)$;
for $N=2$ these quantities are equal.

\begin{corollary}\label{cor:ball}
For every measurable $\Omega\subset\R^N$ with $0<\abs{\Omega}<\infty$, we have
\eq{
    \Ein(\Omega)&\leq\frac{I(\Omega)}{2N\omega_{N-1}}
    \leq\frac{I(\Omega^*)}{2N\omega_{N-1}}
    =\Ein(\Omega^*)\\
    &=\frac{\abs{\Omega}}{N(N+2)}
    \left(\frac{\abs{\Omega}}{s_N}\right)^{2/N},
}
with equality between the outer terms if and only if $\Omega$ is a ball up to a set of measure zero.
\end{corollary}

Martin's higher-dimensional Ahlfors--Beurling inequality \cite{Ma} is a pointwise estimate for the Teodorescu transform $T_\Omega1$, which satisfies $\abs{T_\Omega1}=\abs{\nabla u_\Omega}$. For $N\geq3$ its extremisers are not balls, in line with the classical fact that a body of given volume exerting maximal gravitational attraction at a point need not be a ball \cite{Kl}. By contrast, Corollary~\ref{cor:ball} yields the sharp mean-square bound
\eq{
    \frac{1}{\abs{\Omega}}\int_\Omega\abs{T_\Omega1}^2\,\rd x
    \leq\frac{1}{N(N+2)}\left(\frac{\abs{\Omega}}{s_N}\right)^{2/N}.
}
Integrating Martin's pointwise bound gives the same conclusion only with a larger constant, by the factor $c_N^2(N+2)/N$ (about $1.75$ when $N=3$); see Proposition~\ref{prop:compare}.

We also derive consequences for torsional rigidity and Bergman analytic content. For an open set $\Omega$ of finite measure, define
\eq{
    T(\Omega)=\sup_{\varphi\in C_c^\infty(\Omega)}
    \left(2\int_\Omega\varphi\,\rd x-\int_\Omega\abs{\nabla\varphi}^2\,\rd x\right).
}
When the torsion function $w_\Omega\in W^{1,2}_0(\Omega)$ exists (the weak solution of $-\Delta w_\Omega=1$), one has $T(\Omega)=\int_\Omega w_\Omega\,\rd x=\int_\Omega\abs{\nabla w_\Omega}^2\,\rd x$.

\begin{corollary}\label{cor:SV}
Let $\Omega\subset\R^N$ be open with $0<\abs{\Omega}<\infty$. Then
\eq{
    T(\Omega)\leq\Ein(\Omega)\leq\frac{I(\Omega)}{2N\omega_{N-1}}
    \leq\frac{I(\Omega^*)}{2N\omega_{N-1}}=T(\Omega^*).
}
If $T(\Omega)=T(\Omega^*)$, then $\Omega$ is a ball up to a null set. For $N=2$ the proof uses no rearrangement. For $N\geq3$ it also yields the scale-free bound $T(\Omega)\leq (N-2)\E(\Omega)/(2N)$.
\end{corollary}

The Saint-Venant inequality is classical; standard proofs use symmetrisation \cite{Polya48,PS}. In our approach, the only rearrangement input for $N\geq3$ is for the Newtonian interaction energy. The first inequality in Corollary~\ref{cor:SV} follows by orthogonal projection of $-\nabla u_\Omega$ onto the closed subspace $\{\nabla\varphi:\varphi\in W^{1,2}_0(\Omega)\}$.

For a bounded open set $\Omega\subset\C$, let $A^2(\Omega)$ be its Bergman space
and define
\eq{
    \lambda_{A^2}(\Omega)=\inf_{f\in A^2(\Omega)}
    \norm{\bar z-f}_{L^2(\Omega)}.
}

\begin{corollary}\label{cor:bergman}
Let $\Omega\subset\C$ be a nonempty bounded open set. Then
\eq{
    \lambda_{A^2}(\Omega)^2=4T(\Omega)
    \leq\int_\Omega\abs{\Cc_\Omega}^2\,\rdA
    \leq\frac{\abs{\Omega}^2}{2\pi}.
}
If $\lambda_{A^2}(\Omega)^2=\abs{\Omega}^2/(2\pi)$, then $\Omega$ is a disc up to a set of measure zero.
\end{corollary}

For bounded simply connected domains, the identity $\lambda_{A^2}(\Omega)^2=4T(\Omega)$ was proved by Fleeman--Lundberg \cite{FL}; see also \cite{FK}. For smooth domains, the corresponding identity in every dimension follows from \cite[Theorem~16]{GGS20}. Here $T(\Omega)$ is defined using zero Dirichlet data on the whole boundary. For multiply connected domains, this differs from the torsional rigidity with constant values and prescribed fluxes on inner boundary components considered in \cite{FL,GGS20}. We give a direct projection proof for arbitrary bounded open sets in Section~\ref{sec:cons}.

The higher-dimensional analogue is as follows. For bounded open $\Omega\subset\R^N$, let
\eq{
    \mathcal{A}^2(\Omega)=\{F\in L^2(\Omega;\R^N):
    \rdiv F=0,\ \p_iF_j=\p_jF_i\},\\
    \lambda_{\mathcal{A}^2}(\Omega)=
    \inf_{F\in\mathcal{A}^2(\Omega)}\norm{x-F}_{L^2(\Omega)},
}
where the derivative conditions are distributional. Thus $\mathcal{A}^2$ consists of divergence-free and curl-free vector fields.

\begin{corollary}\label{cor:bergmanND}
Let $N\geq2$ and let $\Omega\subset\R^N$ be a nonempty bounded open set. Then
\eq{
    \lambda_{\mathcal{A}^2}(\Omega)^2=N^2T(\Omega)
    \leq\frac{N}{N+2}\,s_N^{-2/N}\abs{\Omega}^{1+2/N}.
}
Equality in the bound holds only if $\Omega$ is a ball up to a set of measure zero.
For smooth bounded domains, equality holds if and only if $\Omega$ is a ball.
\end{corollary}

For $N=2$, this recovers Corollary~\ref{cor:bergman} under the standard identification with holomorphic functions. Gardiner--Ghergu--Sj\"odin \cite{GGS20} proved the sharp bound for smoothly bounded domains and, more generally, for Bergman $p$-analytic content for $1\leq p\leq2$.

We also obtain applications to Brownian motion and to aggregation equations. Let $B_t$ be Brownian motion with generator $\tfrac{1}{2}\Delta$, let $\tau_\Omega$ be the first exit time from $\Omega$, and let $L_\Omega=\int_0^\infty\chi_\Omega(B_t)\,\rd t$ be the total occupation time.

\begin{theorem}\label{thm:bm}
Let $N\geq3$ and let $\Omega\subset\R^N$ be open with $0<\abs{\Omega}<\infty$. If $B_0$ is uniformly distributed in $\Omega$, then
\eq{
    \mathbb{E}[\tau_\Omega]\leq\frac{N-2}{2N}\mathbb{E}[L_\Omega].
}
Equality holds if and only if $\Omega$ is a ball up to a set of capacity zero.
\end{theorem}

Both expectations are finite. In $\R^3$, the expected occupation time after the first exit is at least five sixths of the total expected occupation time. This ratio statement does not follow from the corresponding fixed-volume bounds for $\mathbb{E}[\tau_\Omega]$ and $\mathbb{E}[L_\Omega]$ separately.

Consider next the aggregation equation
\eq{\label{eq:agg}
    \p_t\rho+\rdiv (\rho v)=0,
    \qquad v=\mp\nabla U_\rho,\qquad\Delta U_\rho=\rho,
}
with the upper sign for aggregation and the lower sign for spreading. For $\rho_0>0$, uniform patches $\rho_0\chi_{\Omega_0}$ evolve as $\rho(t)\chi_{\Omega_t}$ \cite{BLLe,BGLV}. For $N\geq3$, the ratio $R(t)=\E(\Omega_t)/\E(\Omega_t^*)\in(0,1]$ satisfies
\eq{
    \frac{\rd}{\rd t}\log R(t)
    =\mp2\rho(t)\frac{\mathcal{D}(\Omega_t)}{\E(\Omega_t)}.
}
In particular, $R$ is nonincreasing along aggregation and nondecreasing along spreading, with strict monotonicity unless the patch is a ball. The planar analogue involves logarithmic energy; see Theorem~\ref{thm:agg}. For spreading patches, convergence in $L^1$ to the self-similar ball was proved in \cite{BLLe}.

Finally, the potential estimate extends from characteristic functions to bounded nonnegative densities (Theorem~\ref{thm:densities}). Let $\Gamma$ be the fundamental solution of $\Delta$ and set $U_f=\Gamma*f$, with the planar normalization understood as for $u_\Omega$.

\begin{theorem}\label{thm:tri}
Let $N\geq2$ and $0\leq f\in L^1\cap L^\infty(\R^N)$. Then
\eq{
    \int_{\R^N}f\abs{\nabla U_f}^2\,\rd x
    \leq\frac{\norm{f}_{L^\infty}}{2N\omega_{N-1}}
    \iint_{\R^N\times\R^N}
    \frac{f(x)f(y)}{\abs{x-y}^{N-2}}\,\rd x\,\rd y,
}
with equality if and only if $f=c\chi_B$ a.e.\ for a ball $B$ and $c\geq0$. For $N\geq3$ the right-hand side equals $\frac{N-2}{2N}\norm{f}_{L^\infty}\norm{f}_{\dot H^{-1}}^2$. For $N=2$ it equals $\norm{f}_{L^\infty}\norm{f}_{L^1}^2/(8\pi)$; equivalently,
\eq{
    \int_\C f\abs{\Cc f}^2\,\rdA
    \leq\frac{1}{2\pi}\norm{f}_{L^\infty}\norm{f}_{L^1}^2,
    \qquad
    \Cc f(\zeta)=\frac{1}{\pi}\int_\C\frac{f(z)\,\rdA(z)}{\zeta-z}.
}
\end{theorem}

For $N\geq3$, H\"older's inequality gives the weaker constant $1$ in place of $(N-2)/(2N)$. For signed densities, with $\abs{f}$ on the left-hand side, that constant is attained (Section~\ref{sec:densities}). The trilinear integral is the dissipation of the interaction energy along \eqref{eq:agg}, and it yields the sharp bounds in Corollary~\ref{cor:energybounds}.

To our knowledge, Theorem~\ref{thm:main}, the higher-dimensional statement in Corollary~\ref{cor:ball}, and the scale-invariant estimate $T(\Omega)\leq I(\Omega)/(2N\omega_{N-1})$ are new. Related geometric inequalities proved without symmetrisation include Cabr\'e's isoperimetric argument \cite{Ca} and Borell's Brunn--Minkowski inequality for torsional rigidity \cite{Bo}. The constant $(N-2)/(2N)$ also appears in the recent Neumann--Poincar\'e eigenvalue bound of Colbrook and Sadeghi \cite{CS}; we do not know of a direct implication between the two results. Our final rigidity step is related to the potential-theoretic theorem of Aharonov--Schiffer--Zalcman \cite{ASZ}.

Section~\ref{sec:proof} proves the deficit identity, the main inequality, and the equality case. Section~\ref{sec:cons} proves the corollaries and develops further consequences and examples. Section~\ref{sec:martin} compares the pointwise and mean-square bounds. Sections~\ref{sec:densities} and~\ref{sec:apps} treat densities and applications.

\section{Proof of Theorems~\ref{thm:main} and~\ref{thm:deficit}}\label{sec:proof}

\subsection{Preliminaries}\label{sec:prelim}

We write $\rd x$, $\rd p$, and $\rd t$ for Lebesgue measure on $\R^N$, hyperplanes, and lines, respectively, and $\rd\theta$ for surface measure on $\Sph$. When taking line and ray sections, we may replace $\Omega$ by a Borel representative. This leaves the potential, energies, and integrated deficit terms unchanged. We use Tonelli's theorem for nonnegative integrands and Fubini's theorem for absolutely integrable ones.

Let $\Omega^*$ be the open ball centered at $0$ with $\abs{\Omega^*}=\abs{\Omega}$, and let $r_\Omega$ be its radius. For fixed $\theta\in\Sph$, write $x=p+t\theta$ with $p\in\theta^\perp$ and $t\in\R$, so that $\rd x=\rd p\,\rd t$.
The fundamental solution of $\Delta$ is
\eq{
    \Gamma(x)=\frac{1}{2\pi}\log\abs{x} \qquad\text{for}\quad N=2,
}
and
\eq{
    \Gamma(x)=-\abs{x}^{2-N}/((N-2)\omega_{N-1}) \qquad\text{for}\quad N\geq3.
}
In both cases we have
\eq{
    \nabla\Gamma(x)=\frac{x}{\omega_{N-1}\abs{x}^N}.
}

\begin{lemma}\label{lem:pot}
Let $\Omega\subset\R^N$ ($N\geq2$) be measurable with $0<\abs{\Omega}<\infty$.
\begin{enumerate}[(a)]
\item For every $x\in\R^N$,
\eq{
    \int_\Omega\abs{x-y}^{1-N}\,\rd y\leq \omega_{N-1}r_\Omega, \qquad \int_\Omega\abs{x-y}^{2-N}\,\rd y\leq\frac{1}{2} \omega_{N-1}r_\Omega^2.
}
In particular,
\eq{
    I(\Omega)\leq\frac{1}{2} \omega_{N-1}r_\Omega^2\abs{\Omega}<\infty, \qquad \abs{\nabla u_\Omega}\leq r_\Omega.
}
\item $u_\Omega\in C^1(\R^N)$, \eqref{eq:grad} holds, and $\Delta u_\Omega=\chi_\Omega$ in $\mathcal{D}'(\R^N)$.
\item If $N\geq3$, then $\nabla u_\Omega\in L^2(\R^N)$ and
\eq{\label{eq:total-energy}
    \E(\Omega)=\int_{\R^N}\abs{\nabla u_\Omega}^2\,\rd x=\frac{I(\Omega)}{(N-2)\omega_{N-1}}.
}
\item If $\Omega$ is open, then $u_\Omega-\abs{x}^2/(2N)$ is harmonic in $\Omega$; in particular $u_\Omega\in C^\infty(\Omega)$. Moreover, $u_\Omega$ is harmonic in $\R^N\backslash\overline\Omega$.
\end{enumerate}
\end{lemma}

\begin{proof}
(a) By the bathtub principle \cite[Thm.~1.14]{LL}, each integral is maximized by the ball of volume $\abs{\Omega}$ centered at $x$. Integration over that ball gives
$$\int_{B_r}\abs{y}^{1-N}\,\rd y=\omega_{N-1}r, \qquad \int_{B_r}\abs{y}^{2-N}\,\rd y=\frac 12 \omega_{N-1}r^2.$$ For $N=2$ the second bound is an identity. Integrating in $x$ gives the bound on $I(\Omega)$; the gradient estimate follows from \eqref{eq:grad}, proved below.

(b) Write
\eq{
    k(x,y)=\begin{cases}
    \Gamma(x-y), & N\geq3,\\
    \frac{1}{2\pi}\log\frac{\abs{x-y}}{1+\abs{y}}, & N=2.
    \end{cases}
}
In both cases $\nabla_x k(x,y)=\nabla\Gamma(x-y)$. Fix $\rho>0$ and split
$\chi_\Omega=\chi_{\Omega\cap B_{2\rho}}+\chi_{\Omega\backslash B_{2\rho}}$.
The potential of $\chi_{\Omega\cap B_{2\rho}}$ differs from $\Gamma*\chi_{\Omega\cap B_{2\rho}}$ by a constant and is $C^1$, with gradient given by \eqref{eq:grad} \cite[Lemma~4.1]{GT}.

If $x\in B_\rho$ and $y\notin B_{2\rho}$, then $\tfrac{1}{2}\abs{y}\leq\abs{x-y}\leq\tfrac{3}{2}\abs{y}$, so $k(x,y)$ and $\nabla_xk(x,y)$ are uniformly bounded on $B_\rho\times(\R^N\backslash B_{2\rho})$. Since $\abs{\Omega}<\infty$, differentiation under the integral is justified on $B_\rho$. As $\rho>0$ is arbitrary, this proves $u_\Omega\in C^1(\R^N)$ and \eqref{eq:grad}.

To verify the distributional identity $\Delta u_\Omega=\chi_\Omega$, let $\phi\in C_c^\infty(\R^N)$. Integration by parts gives
\eq{
    \int_{\R^N} u_\Omega\,\Delta\phi\,\rd x=-\int_{\R^N}\nabla u_\Omega\cdot\nabla\phi\,\rd x.
}
By \eqref{eq:grad}, part (a), and Fubini's theorem,
\eq{
    -\int_{\R^N}\nabla u_\Omega\cdot\nabla\phi\,\rd x
    &= -\int_\Omega\int_{\R^N}\nabla\Gamma(x-y)\cdot\nabla\phi(x)\,\rd x\,\rd y
    \\&=\int_\Omega\phi(y)\,\rd y,
}
using $\Delta\Gamma=\delta_0$. Hence $\Delta u_\Omega=\chi_\Omega$ in $\mathcal{D}'(\R^N)$.

(c) Assume $N\geq3$. Let $f$ be bounded with compact support, and set $u_f=\Gamma*f$. By Green's formula on $B_\rho$ (after mollification) and the decay estimates
$u_f=O(\abs{x}^{2-N})$ and $\nabla u_f=O(\abs{x}^{1-N})$ as $\abs{x}\to\infty$,
\eq{
    \int_{\R^N}\abs{\nabla u_f}^2\,\rd x=-\int_{\R^N}f\,u_f\,\rd x.
}
By polarization this implies
\eq{
    \int_{\R^N}\nabla u_f\cdot\nabla u_g\,\rd x
    =\frac{1}{(N-2)\omega_{N-1}}\iint_{\R^N\times\R^N} f(x)g(y)\abs{x-y}^{2-N}\,\rd x\,\rd y.
}
Apply this with $f=g=\chi_{\Omega_k}$, where $\Omega_k\coloneqq \Omega\cap B_k$. For $j<k$,
\eq{
    \int_{\R^N}\abs{\nabla u_{\Omega_k}-\nabla u_{\Omega_j}}^2\,\rd x
    &=\frac{1}{(N-2)\omega_{N-1}}\iint_{(\Omega_k\backslash\Omega_j)^2}\abs{x-y}^{2-N}\,\rd x\,\rd y
    \\&\leq \frac{1}{(N-2)\omega_{N-1}}\iint_{(\Omega\backslash\Omega_j)^2}\abs{x-y}^{2-N}\,\rd x\,\rd y.
}
The right-hand side tends to $0$ as $j\to\infty$ by part (a) and dominated convergence, so $\nabla u_{\Omega_k}$ is a  Cauchy seqeunce in $L^2(\R^N)$.
Moreover, $\nabla u_{\Omega_k}(x)\to\nabla u_\Omega(x)$ pointwise by dominated convergence in \eqref{eq:grad}. Therefore $\nabla u_\Omega\in L^2(\R^N)$ and, by monotone convergence,
\eq{
    \int_{\R^N}\abs{\nabla u_\Omega}^2\,\rd x
    =\lim_{k\to\infty}\frac{I(\Omega_k)}{(N-2)\omega_{N-1}}
    =\frac{I(\Omega)}{(N-2)\omega_{N-1}}.
}
This proves \eqref{eq:total-energy}.

(d) The distributional identity gives $\Delta(u_\Omega-\abs{x}^2/(2N))=0$ in $\Omega$, and Weyl's lemma yields the interior assertion. The same argument gives harmonicity in $\R^N\backslash\overline\Omega$.
\end{proof}

\begin{lemma}\label{lem:rays}
For every $x\in\R^N$ the function $\lambda(x,\cdot)$ is integrable on $\Sph$, and
\eq{
    \int_{\Sph}\lambda(x,\theta)\,\rd\theta=\int_\Omega\abs{x-y}^{1-N}\,\rd y.
}
Moreover,
\eq{
    \nabla u_\Omega(x)
    =-\frac{1}{\omega_{N-1}}\int_{\Sph}\theta\,\lambda(x,\theta)\,\rd\theta
    =-\frac{1}{2\omega_{N-1}}\int_{\Sph}\theta\,\delta(x,\theta)\,\rd\theta.
}
\end{lemma}

\begin{proof}
In polar coordinates $y=x+r\theta$ one has $\rd y=r^{N-1}\,\rd r\,\rd\theta$ and $(x-y)\abs{x-y}^{-N}=-\theta\,r^{1-N}$. Hence
\eq{
    \int_\Omega\frac{\rd y}{\abs{x-y}^{N-1}}
    =\int_{\Sph}\int_0^\infty\chi_\Omega(x+r\theta)\,\rd r\,\rd\theta
    =\int_{\Sph}\lambda(x,\theta)\,\rd\theta,
}
which is finite by Lemma~\ref{lem:pot}(a). Using the same change of variables in \eqref{eq:grad} gives the first expression for $\nabla u_\Omega(x)$. In fact,
\eq{%\label{eq:grad}
  \omega_{N-1}   \nabla u_\Omega(x)=\int_\Omega
    \frac{x-y}{\abs{x-y}^N}\,\rd y = 
    - \int_{\Sph}\int_0^\infty\chi_\Omega(x+r\theta)\theta \,\rd r\,\rd\theta
    =\int_{\Sph}\theta\lambda(x,\theta)\,\rd\theta,
}
Changing $\theta$ to $-\theta$ gives
\eq{
    \int_{\Sph}\theta\,\lambda(x,\theta)\,\rd\theta
    =\frac{1}{2}\int_{\Sph}\theta\,\big(\lambda(x,\theta)-\lambda(x,-\theta)\big)\,\rd\theta
    =\frac{1}{2}\int_{\Sph}\theta\,\delta(x,\theta)\,\rd\theta.
}
\end{proof}

\begin{lemma}\label{lem:1d}
Let $E\subset\R$ be measurable with $\sigma=\abs{E}<\infty$, and set $F(t)=\abs{E\cap(-\infty,t)}$.
\begin{enumerate}[(a)]
\item For every Borel function $\varphi\geq0$ on $[0,\sigma]$,
\eq{
    \int_E\varphi\big(F(t)\big)\,\rd t=\int_0^\sigma\varphi(s)\,\rd s.
}
\item One has $\int_E(\sigma-2F(t))^2\,\rd t=\sigma^3/3$.
\item One has $\G(E)\geq0$. More precisely,
\eq{
    \G(E)=\iint_{E\times E}\Big(\abs{s-t}-\abs{F(s)-F(t)}\Big)\,\rd s\,\rd t,
}
where for $s<t$ the integrand equals $\abs{(s,t)\backslash E}\geq0$. Moreover, $\G(E)=0$ if and only if $E$ is an interval up to a set of measure zero, where the empty interval is allowed.
\end{enumerate}
\end{lemma}

\begin{proof}
The conclusions are immediate if $\sigma=0$. Assume $\sigma>0$.

(a) The function $F$ is nondecreasing and $1$-Lipschitz, with $F'=\chi_E$ a.e.. The change-of-variables formula for absolutely continuous monotone functions gives
\eq{
    \int_E\varphi(F(t))\,\rd t
    =\int_\R\varphi(F(t))F'(t)\,\rd t
    =\int_0^\sigma\varphi(s)\,\rd s.
}

(b) Apply (a) with $\varphi(s)=(\sigma-2s)^2$.

(c) Applying part (a) in each variable,
\eq{
    \iint_{E\times E}\abs{F(s)-F(t)}\,\rd s\,\rd t
    =\iint_{[0,\sigma]^2}\abs{a-b}\,\rd a\,\rd b
    =\frac{\sigma^3}{3}.
}
This gives the formula for $\G(E)$. For $s<t$,
\eq{
    \abs{s-t}-\abs{F(s)-F(t)}=(t-s)-(F(t)-F(s))=\abs{(s,t)\backslash E}\geq0.
}

If $E$ is an interval up to a set of measure zero, then $\abs{(s,t)\backslash E}=0$ for a.e.\ $s<t$ in $E$, so the integrand vanishes a.e.\ on $E\times E$.
Conversely, suppose that $E$ does not agree almost everywhere with an interval, and let $E'\subset E$ be its density points. The interior of the convex hull of $E'$ is a countable union of intervals with endpoints in $E'$. Since this convex hull does not agree almost everywhere with $E$, there are $s_0<t_0$ in $E'$ such that
$\eta\coloneqq \abs{(s_0,t_0)\backslash E}>0$.
For $0<\varepsilon<\min\{\eta/4,(t_0-s_0)/4\}$, all
$s\in E\cap(s_0-\varepsilon,s_0+\varepsilon)$ and $t\in E\cap(t_0-\varepsilon,t_0+\varepsilon)$ satisfy $\abs{(s,t)\backslash E}\geq\eta-2\varepsilon>0$, and these pairs have positive measure. Therefore $\G(E)>0$, a contradition.
\end{proof}

\begin{lemma}\label{lem:BP}
For every measurable function $\Phi\geq0$ on $\R^N\times\R^N$,
\eq{
    \iint_{\R^N\times\R^N}\Phi(x,y)\,\rd x\,\rd y
    =\frac{1}{2}\int_{\Sph}\int_{\theta^\perp}\iint_{\R^2}\Phi(p+s\theta,p+t\theta)\,\abs{s-t}^{N-1}\,\rd s\,\rd t\,\rd p\,\rd\theta.
}
Consequently,
\eq{\label{eq:chordpower}
    I(\Omega)
    &=\frac{1}{2}\int_{\Sph}\int_{\theta^\perp}\iint_{E_{\theta,p}\times E_{\theta,p}}\abs{s-t}\,\rd s\,\rd t\,\rd p\,\rd\theta\\
    &=\frac{1}{6}\int_{\Sph}\int_{\theta^\perp}\sigma(\theta,p)^3\,\rd p\,\rd\theta
    +\frac{1}{2}\int_{\Sph}\int_{\theta^\perp}\G(E_{\theta,p})\,\rd p\,\rd\theta.
}
\end{lemma}

\begin{proof}
Write $y=x+\rho\theta$ in polar coordinates, and then write $x=p+s\theta$ with $p\in\theta^\perp$. Then
\eq{
    \iint_{\R^N\times\R^N}\Phi(x,y)\,\rd x\,\rd y
    &=\int_{\Sph}\int_0^\infty\!\!\int_{\R^N}\Phi(x,x+\rho\theta)\,\rd x\,\rho^{N-1}\,\rd\rho\,\rd\theta\\
    &=\int_{\Sph}\int_{\theta^\perp}\iint_{\{t>s\}}\Phi(p+s\theta,p+t\theta)\,(t-s)^{N-1}\,\rd s\,\rd t\,\rd p\,\rd\theta.
}
Changing $(\theta,s,t)$ to $(-\theta,-s,-t)$ gives the same expression over $\{t<s\}$. Averaging proves the formula.

For \eqref{eq:chordpower}, take $\Phi(x,y)=\chi_\Omega(x)\chi_\Omega(y)\abs{x-y}^{2-N}$. On each line $p+\R\theta$ one has $\abs{x-y}=\abs{s-t}$, so
$\abs{x-y}^{2-N}\abs{s-t}^{N-1}=\abs{s-t}$. Now use the definition of $\G$, \eqref{G}. Finally, note that $\sigma(\theta,p)<\infty$ for a.e.\ $(\theta,p)$, since $\int_{\theta^\perp}\sigma(\theta,p)\,\rd p=\abs{\Omega}$.
\end{proof}

The interaction kernel leaves the factor $\abs{s-t}$ in the line formula, while the gradient kernel cancels the polar Jacobian in Lemma~\ref{lem:rays}. For $N=2$, \eqref{eq:chordpower} reads
\eq{
    \abs{\Omega}^2=\frac{1}{2}\int_{\S^1}\!\int_{\theta^\perp}\iint_{E_{\theta,p}\times E_{\theta,p}}\abs{s-t}\,\rd s\,\rd t\,\rd p\,\rd\theta.
}

\subsection{The inequality and the deficit identity}
\begin{proof}[Proof of Theorem~\ref{thm:deficit}]
Since
$\int_{\Sph}\theta_i\theta_j\,\rd\theta=\frac{\omega_{N-1}}{N}\delta_{ij}$,
the coordinate functions $\theta_1,\dots,\theta_N$ are orthogonal in $L^2(\Sph)$ with squared norms $\omega_{N-1}/N$. Hence, for $f\in L^2(\Sph)$,
\eq{
    P_1f=\frac{N}{\omega_{N-1}}\sum_{i=1}^N\<f,\theta_i\>\,\theta_i,
    \qquad
    \norm{P_1f}^2=\frac{N}{\omega_{N-1}}\Abs{\int_{\Sph}\theta\,f(\theta)\,\rd\theta}^2.
}
We use the same formula to define $P_1f$ for $f\in L^1(\Sph)$.

By Lemma~\ref{lem:rays}, for every $x\in\R^N$ we have
\eq{\label{eq:P1}
    P_1\delta(x,\cdot)(\theta)=\frac{N}{\omega_{N-1}}\Big\<\int_{\Sph}\psi\,\delta(x,\psi)\,\rd\psi,\theta\Big\>
    &=-2N\<\nabla u_\Omega(x),\theta\>,\\
    \abs{\nabla u_\Omega(x)}^2&=\frac{1}{4N\omega_{N-1}}\norm{P_1\delta(x,\cdot)}^2.
}
Thus, whenever $\delta(x,\cdot)\in L^2(\Sph)$, orthogonality gives
\eq{\label{eq:pointwise}
    \abs{\nabla u_\Omega(x)}^2
    =\frac{1}{4N\omega_{N-1}}\Big(\norm{\delta(x,\cdot)}^2-\norm{\delta(x,\cdot)-P_1\delta(x,\cdot)}^2\Big).
}

Fix $\theta\in\Sph$ and $p\in\theta^\perp$ with $\sigma=\sigma(\theta,p)<\infty$, and set $F(t)=\abs{E_{\theta,p}\cap(-\infty,t)}$. For $x=p+t\theta\in\Omega$ (i.e.\ $t\in E_{\theta,p}$),
\eq{
    \lambda(x,\theta)=\sigma-F(t),
    \qquad
    \lambda(x,-\theta)=F(t),
    \qquad
    \delta(x,\theta)=\sigma-2F(t).
}
By Lemma~\ref{lem:1d}(b),
$\int_{E_{\theta,p}}\delta(p+t\theta,\theta)^2\,\rd t=\sigma^3/3$.
Integrating in $p\in\theta^\perp$ and $\theta\in\Sph$ gives
\eq{\label{eq:lines}
    \int_\Omega\norm{\delta(x,\cdot)}^2\,\rd x
    =\frac{1}{3}\int_{\Sph}\int_{\theta^\perp}\sigma(\theta,p)^3\,\rd p\,\rd\theta
    \leq 2I(\Omega)<\infty,
}
where the bound follows from \eqref{eq:chordpower} and Lemma~\ref{lem:pot}(a). In particular, $\delta(x,\cdot)\in L^2(\Sph)$ for a.e.\ $x\in\Omega$.
Integrating \eqref{eq:pointwise} over $\Omega$ and using \eqref{eq:lines} yields \eqref{eq:id1}. Dividing \eqref{eq:chordpower} by $2N\omega_{N-1}$ gives \eqref{eq:id2}. Their difference is \eqref{eq:deficit}.
\end{proof}

We now prove Theorem  \ref{thm:main}.

\begin{proof}[Proof of Theorem~\ref{thm:main}]
Both terms on the right-hand side of \eqref{eq:deficit} are nonnegative, the second by Lemma~\ref{lem:1d} (c). This proves \eqref{eq:main}.

For $N\geq3$, part~(ii) follows from \eqref{eq:total-energy}. For $N=2$, identifying vectors with complex numbers in \eqref{eq:grad} gives $\Cc_\Omega=2\overline{\nabla u_\Omega}$. Thus part~(i) follows from $I(\Omega)=\abs{\Omega}^2$.

Let $B$ be a ball centered at $c$. A line through $x\in B$ in direction $\theta$ meets $B$ in a segment whose midpoint $m$ satisfies $\<m-c,\theta\>=0$, hence
\eq{
    \delta(x,\theta)=2\<m-x,\theta\>=-2\<x-c,\theta\>
    \qquad(x\in B,\ \theta\in\Sph).
}
Thus $\delta(x,\cdot)=P_1\delta(x,\cdot)$, every line section is an interval, and both deficit terms vanish. By \eqref{eq:P1} we also obtain $\nabla u_B(x)=(x-c)/N$ on $B$. We leave the proof of equality in the next subsection.
\end{proof}
\subsection{The equality case}\label{sec:rigidity}

We use two lemmas to  study the equality case.

\begin{lemma}\label{lem:convex}
Let $\Omega\subset\R^N$ be measurable with $0<\abs{\Omega}<\infty$, and suppose that
\eq{\label{eq:star}
    \mathcal{H}^1\big([x,y]\backslash\Omega\big)=0\qquad\text{for almost every }(x,y)\in\Omega\times\Omega.
}
Then the set $\Omega^{(1)}$ of points of Lebesgue density one of $\Omega$ is open, bounded, and convex, and $\abs{\Omega\triangle\Omega^{(1)}}=0$.
\end{lemma}

\begin{proof}
First of all, by the Lebesgue density theorem, $\abs{\Omega\triangle\Omega^{(1)}}=0$.
Since
\eq{
    \mathcal{H}^1([x,y]\backslash\Omega)=\abs{x-y}\int_0^1\chi_{\R^N\backslash\Omega}((1-\mu)x+\mu y)\,\rd\mu,
}
Tonelli's theorem and \eqref{eq:star} imply
\eq{
    \int_0^1\iint_{\Omega\times\Omega}\chi_{\R^N\backslash\Omega}\big((1-\mu)x+\mu y\big)\,\rd x\,\rd y\,\rd\mu=0.
}
Let $M\subset(0,1)$ be the set of $\mu$ for which the inner double integral vanishes. Then $M$ has full measure in $(0,1)$.

For $\mu\in(0,1)$ define $h_\mu=\chi_{(1-\mu)\Omega}*\chi_{\mu\Omega}$, so that
\eq{
    h_\mu(w)=\abs{(1-\mu)\Omega\cap(w-\mu\Omega)}.
}
The function $h_\mu$ is continuous, being the convolution of an integrable function with a bounded one. If $\mu\in M$, then
\eq{
    \int_{\R^N\backslash\Omega}h_\mu(w)\,\rd w
    =(1-\mu)^N\mu^N\iint_{\Omega\times\Omega}\chi_{\R^N\backslash\Omega}\big((1-\mu)x+\mu y\big)\,\rd x\,\rd y
    =0.
}
Consequently, the open set $\{h_\mu>0\}$ is contained in $\Omega$ up to a null set, and hence in $\Omega^{(1)}$.

Let $x,y\in\Omega^{(1)}$, fix $\mu\in(0,1)$, and set $w=(1-\mu)x+\mu y$ and $\kappa=(1-\mu)/\mu$. Substituting $a=(1-\mu)(x+\xi)$ in the definition of $h_\mu$ gives $w-a=\mu(y-\kappa\xi)$. Therefore, for every $\rho>0$,
\eq{
    h_\mu(w)
    &=(1-\mu)^N\abs{\{\xi\in\R^N:\ x+\xi\in\Omega,\ y-\kappa\xi\in\Omega\}}\\
    &\geq (1-\mu)^N\Big(\abs{B_\rho}-\abs{B_\rho(x)\backslash\Omega}-\kappa^{-N}\abs{B_{\kappa\rho}(y)\backslash\Omega}\Big).
}
Since $x$ and $y$ are density points of $\Omega$, the bracket is positive for sufficiently small $\rho$, hence $h_\mu(w)>0$.

Taking $y=x$ and $\mu\in M$ shows that every $x\in\Omega^{(1)}$ belongs to the open set $\{h_\mu>0\}\subset\Omega^{(1)}$. Thus $\Omega^{(1)}$ is open. It is nonempty by the density theorem.

For convexity, let $x,y\in\Omega^{(1)}$ and $\tau\in(0,1)$. Choose $\varepsilon>0$ with $B_\varepsilon(x)\cup B_\varepsilon(y)\subset\Omega^{(1)}$, and pick $\mu\in M$ such that $\abs{\mu-\tau}\,\abs{y-x}<\varepsilon$. Then
$x'=x+(\tau-\mu)(y-x)$ and $y'=y+(\tau-\mu)(y-x)$ lie in $\Omega^{(1)}$, and
$(1-\mu)x'+\mu y'=x+\tau(y-x)$. The positivity argument gives $h_\mu((1-\mu)x'+\mu y')>0$, so this point lies in $\Omega^{(1)}$. Thus $\Omega^{(1)}$ is convex.

Finally, any open convex set of finite measure is bounded: if it contains $B_\varepsilon(a)$ and a point $b$, then it contains the convex hull of $B_\varepsilon(a)\cup\{b\}$, whose volume is bounded below by a positive multiple of $\varepsilon^{N-1}\abs{b-a}$.
\end{proof}

\begin{lemma}\label{lem:convexrays}
Let $K\subset\R^N$ be open, bounded, and convex, and define $\lambda$ and $\delta$ with $K$ in place of $\Omega$. For $x\in K$ and $\theta\in\Sph$, the line through $x$ in direction $\theta$ meets $K$ in the open segment
\eq{
    \{x+r\theta:\ -\lambda(x,-\theta)<r<\lambda(x,\theta)\}.
}
Moreover, $\lambda$ is continuous on $K\times\Sph$. In particular, if $x,\,x+r\theta\in K$, then
\eq{\label{eq:shift}
    \delta(x+r\theta,\theta)=\delta(x,\theta)-2r.
}
\end{lemma}

\begin{proof}
The intersection of $K$ with a line is convex, open (as a subset of that line), and bounded, which gives the segment description and \eqref{eq:shift}.

For lower semicontinuity of $\lambda$ at $(x,\theta)$, let $0<r<\lambda(x,\theta)$. Then $x+r\theta\in K$, hence $x'+r\theta'\in K$ for all $(x',\theta')$ sufficiently close to $(x,\theta)$, and therefore $\lambda(x',\theta')\geq r$.

For upper semicontinuity, let $r>\lambda(x,\theta)$. Then $x+r\theta\notin\overline K$; otherwise, since $x$ is an interior point of the convex set $\overline K$, the half-open segment $[x,x+r\theta)$ would lie in $K$, giving $\lambda(x,\theta)\geq r$. By closedness, there is a neighborhood of $(x,\theta)$ such that $x'+r\theta'\notin\overline K$ throughout that neighborhood. If $\lambda(x',\theta')\geq r$, then $[x',x'+r\theta']\subset\overline K$, hence $x'+r\theta'\in\overline K$, a contradiction. Therefore $\lambda(x',\theta')<r$, proving upper semicontinuity.
\end{proof}

\begin{theorem}\label{thm:rigidity}
Let $\Omega\subset\R^N$, $N\geq2$, be measurable with $0<\abs{\Omega}<\infty$. If equality holds in \eqref{eq:main}, then $\Omega$ coincides with a ball up to a set of measure zero.
\end{theorem}

\begin{proof}
Both terms in \eqref{eq:deficit} vanish. In particular, $\G(E_{\theta,p})=0$ for a.e.\ $(\theta,p)$. By Lemma~\ref{lem:1d}(c), this implies that for such $(\theta,p)$ the set $E_{\theta,p}$ is an interval up to a set of measure zero. Hence $\mathcal{H}^1([x,y]\backslash\Omega)=0$ for almost every pair of points in almost every such line section.

Let $\mathcal{B}\subset\Omega\times\Omega$ be the set of pairs $(x,y)$ such that $\mathcal{H}^1([x,y]\backslash\Omega)>0$. Applying Lemma~\ref{lem:BP} with $\Phi=\chi_{\mathcal{B}}$ yields $\abs{\mathcal{B}}=0$, which is exactly \eqref{eq:star}. By Lemma~\ref{lem:convex}, the set $K\coloneqq \Omega^{(1)}$ is open, bounded, and convex, and $\abs{\Omega\triangle K}=0$.

Replacing $\Omega$ by $K$ does not change $\nabla u_\Omega$, $\Ein(\Omega)$, or $I(\Omega)$, so equality holds in \eqref{eq:main} for $K$ as well.
From now on $\lambda$ and $\delta$ refer to $K$, and we write $u=u_K$.

The vanishing of the first deficit term gives $\delta(x,\cdot)=P_1\delta(x,\cdot)$ in $L^2(\Sph)$ for a.e.\ $x\in K$. By \eqref{eq:P1},
\eq{\label{eq:linear}
    \delta(x,\theta)=-2N\<\nabla u(x),\theta\>.
}
By Lemmas~\ref{lem:convexrays} and~\ref{lem:pot}(b), both sides are continuous on $K\times\Sph$. For a.e.\ $x\in K$ the identity holds for a.e.\ $\theta$, hence for all $\theta$; the set of such $x$ is dense in $K$. By continuity, \eqref{eq:linear} holds for all $(x,\theta)\in K\times\Sph$.

Fix $x\in K$, $\theta\in\Sph$, and $r$ sufficiently small. By \eqref{eq:shift} and \eqref{eq:linear},
\eq{
    \<\nabla u(x+r\theta)-\nabla u(x),\theta\>=\frac{r}{N}.
}
Since $u\in C^\infty(K)$ by Lemma~\ref{lem:pot}(d), differentiating at $r=0$ gives
$\<D^2u(x)\theta,\theta\>=\frac{1}{N}$ for all $\theta\in\Sph$. As $D^2u(x)$ is symmetric, it follows that $D^2u\equiv\frac{1}{N} I$ on $K$. Since $K$ is connected, there exists $c\in\R^N$ such that
\eq{
    \nabla u(x)=\frac{x-c}{N}\qquad(x\in K).
}
After translating, we may assume $c=0$. By continuity, $\nabla u(x)=x/N$ on $\overline K$.

For $1\leq i<j\leq N$, let $L_{ij}=x_i\p_j-x_j\p_i$, the generator of rotations in the $(x_i,x_j)$-plane. Since $u\in C^1(\R^N)$, the function $L_{ij}u$ is continuous on $\R^N$. The formula for $\nabla u$ shows that it vanishes on $K$, hence on $\overline K$. On $U\coloneqq \R^N\backslash\overline K$ it is harmonic, because $u$ is harmonic there and $\Delta L_{ij}=L_{ij}\Delta$.

Moreover, if $\overline K\subset B_{R_0}$, then
\eq{
    L_{ij}u(x)=\frac{1}{\omega_{N-1}}\int_K\frac{x_jy_i-x_iy_j}{\abs{x-y}^N}\,\rd y
    =O(\abs{x}^{1-N})\qquad(\abs{x}\to\infty).
}
Applying the maximum principle to $\pm L_{ij}u$ on $U\cap B_\rho$ and letting $\rho\to\infty$ yields $L_{ij}u\equiv0$ on $U$, hence on all of $\R^N$.

Let $A$ be antisymmetric. For $x\in\R^N$ one has
$\frac{\rd}{\rd t}u(e^{tA}x)=\<\nabla u(e^{tA}x),Ae^{tA}x\>$, and
$\<\nabla u(z),Az\>=\sum_{i<j}A_{ji}\,L_{ij}u(z)=0$.
It follows that $u(e^{tA}x)=u(x)$ for all $t$, and hence $u\circ R=u$ for every $R\in SO(N)$. Since $\Delta$ commutes with rotations, it follows that $\chi_K\circ R=\chi_K$ a.e., i.e.\ $\abs{R^{-1}K\triangle K}=0$.

Two open convex sets $A,B$ with $\abs{A\triangle B}=0$ coincide: if $A\backslash\overline B$ were nonempty it would contain an open set of positive measure. Thus $A\subset\overline B$, hence $A\subset\operatorname{int}\overline B=B$, and by symmetry $A=B$. Therefore $R^{-1}K=K$ for every rotation $R$.

Since $SO(N)$ acts transitively on spheres centered at $0$, the nonempty open convex set $K$ is a union of such spheres. It contains the convex hull of each sphere it meets, i.e.\ a closed ball. Being open, $K$ is therefore an open ball $B_R(0)$, with $\abs{B_R}=\abs{K}=\abs{\Omega}$. Undoing the translation, we conclude that $\Omega$ is a ball up to a set of measure zero.
\end{proof}

In geometric tomography \cite[Ch.~5]{Ga}, $\lambda(x,\cdot)$ is the directed X-ray of $\Omega$ at $x$. The proof shows that a convex body is a ball if $\delta(x,\cdot)$ is linear in $\theta$ at every interior point.

\section{Proofs of the corollaries and examples}\label{sec:cons}

We prove the corollaries from the introduction and give examples. We first recall the Riesz--Sobolev rearrangement inequality.

A nonnegative measurable function $f$ on $\R^N$ vanishes at infinity if $\abs{\{f>t\}}<\infty$ for every $t>0$. For such $f$, let $f^*$ denote its symmetric decreasing rearrangement. In particular, $(\chi_A)^*=\chi_{A^*}$. We say that $g$ is strictly symmetric decreasing if $g=g^*$ and $g(x)>g(y)$ whenever $\abs{x}<\abs{y}$.

\begin{theorem}\label{thm:RS}
Let $f,g,h$ be nonnegative measurable functions on $\R^N$ vanishing at infinity. Then
\eq{
    \iint f(x)\,g(x-y)\,h(y)\,\rd x\,\rd y
    \leq\iint f^*(x)\,g^*(x-y)\,h^*(y)\,\rd x\,\rd y.
}
If moreover $g$ is strictly symmetric decreasing, $f$ and $h$ are not identically zero, and equality holds with a finite positive value, then $f=f^*(\cdot-a)$ and $h=h^*(\cdot-a)$ almost everywhere, for some $a\in\R^N$.
\end{theorem}

The inequality is due to F.~Riesz in dimension one \cite{Ri} and to Sobolev in higher dimensions \cite{So}; see \cite[Thm.~3.7]{LL}. The equality statement is due to Lieb \cite{Li}; see \cite[Thm.~3.9]{LL} and \cite{Bu} for more general equality results.

For $N\geq3$, taking $f=h=\chi_\Omega$ and $g(z)=\abs{z}^{2-N}$ yields $I(\Omega)\leq I(\Omega^*)$, with equality only for balls up to sets of measure zero. By Lemma~\ref{lem:pot}(c), this is equivalent to $\E(\Omega)\leq\E(\Omega^*)$.

\begin{proof}[Proof of Corollary~\ref{cor:ball}]
The first inequality follows from Theorem~\ref{thm:main}. The second is an equality for $N=2$ and follows from Theorem~\ref{thm:RS} for $N\geq3$.

A ball attains equality in Theorem~\ref{thm:main}, so $I(\Omega^*)/(2N\omega_{N-1})=\Ein(\Omega^*)$. Writing $r=r_\Omega$ and $\omega_{N-1}=Ns_N$,
\eq{
    \Ein(\Omega^*)
    =\int_{B_r}\frac{\abs{x}^2}{N^2}\,\rd x
    =\frac{\omega_{N-1}r^{N+2}}{N^2(N+2)}
    =\frac{\abs{\Omega}\,r^2}{N(N+2)}
    =\frac{\abs{\Omega}}{N(N+2)}\Big(\frac{\abs{\Omega}}{s_N}\Big)^{2/N}.
}
If the outer terms are equal, then equality holds in Theorem~\ref{thm:main}. Theorem~\ref{thm:rigidity} implies that $\Omega$ is a ball up to a set of measure zero. Conversely, the computation above shows that balls attain equality.
\end{proof}

\begin{proof}[Proof of Corollary~\ref{cor:SV}]
Let $\varphi\in C_c^\infty(\Omega)$. Since $u_\Omega\in C^1(\R^N)$ and $\Delta u_\Omega=1$ in $\mathcal{D}'(\Omega)$, integration by parts gives
$\int_\Omega\nabla u_\Omega\cdot\nabla\varphi\,\rd x=-\int_\Omega\varphi\,\rd x$. Hence
\eq{
    2\int_\Omega\varphi\,\rd x-\int_\Omega\abs{\nabla\varphi}^2\,\rd x
    &=-2\int_\Omega\nabla u_\Omega\cdot\nabla\varphi\,\rd x-\int_\Omega\abs{\nabla\varphi}^2\,\rd x\\
    &=\int_\Omega\abs{\nabla u_\Omega}^2\,\rd x-\int_\Omega\abs{\nabla u_\Omega+\nabla\varphi}^2\,\rd x
    \leq\int_\Omega\abs{\nabla u_\Omega}^2\,\rd x.
}
Taking the supremum over $\varphi$ yields the first inequality.
Equivalently, $-\nabla w_\Omega$ is the orthogonal projection of $\nabla u_\Omega$ onto the closure of $\{\nabla\varphi:\varphi\in C_c^\infty(\Omega)\}$ in $L^2(\Omega;\R^N)$.

The next two inequalities are Corollary~\ref{cor:ball}. For a ball $B$ of radius $R$ centered at $0$, we have $w_B=(R^2-\abs{x}^2)/(2N)$ and $\nabla u_B=x/N$ on $B$. Thus
$T(B)=\int_B\abs{x}^2/N^2\,\rd x=\Ein(B)=\frac{1}{2N\omega_{N-1}}I(B)$, which gives the final equality.

If $T(\Omega)=T(\Omega^*)$, then equality holds throughout the chain. Theorem~\ref{thm:rigidity}, applied to the second inequality, gives the conclusion.
\end{proof}

The estimate $T(\Omega)\leq\frac{1}{2N\omega_{N-1}}I(\Omega)$ reads $T(\Omega)\leq\abs{\Omega}^2/(8\pi)$ when $N=2$. For $N\geq3$, combining it with the Kohler--Jobin inequality gives the following lower bound for the Newtonian energy.

\begin{corollary}\label{cor:KJ}
Let $N\geq3$ and let $\Omega\subset\R^N$ be open with $\abs{\Omega}\in(0,\infty)$. Denote by $\Lambda_1(\Omega)$ its first Dirichlet eigenvalue. Then
\eq{
    \E(\Omega)
    \;\geq\;\frac{2N}{N-2}\,T(\Omega)
    \;\geq\;\frac{2N}{N-2}\,T(B_1)\Big(\frac{\Lambda_1(B_1)}{\Lambda_1(\Omega)}\Big)^{\frac{N+2}{2}}
    \;=\;\frac{2\,s_N}{(N-2)(N+2)}\Big(\frac{j_{N/2-1,1}^{\,2}}{\Lambda_1(\Omega)}\Big)^{\frac{N+2}{2}},
}
where $j_{N/2-1,1}$ is the first positive zero of the Bessel function $J_{N/2-1}$. Equality holds throughout if and only if $\Omega$ is a ball up to a set of zero capacity.
\end{corollary}

\begin{proof}
The first inequality is Corollary~\ref{cor:SV} in the form $T(\Omega)\leq\frac{N-2}{2N}\E(\Omega)$.
The second follows from the Kohler--Jobin inequality
$\Lambda_1(\Omega)\,T(\Omega)^{2/(N+2)}\geq\Lambda_1(B_1)\,T(B_1)^{2/(N+2)}$ \cite{KJ,Br}. For the unit ball,
$T(B_1)=\int_{B_1}\abs{x}^2/N^2\,\rd x=s_N/(N(N+2))$ and $\Lambda_1(B_1)=j_{N/2-1,1}^2$.

Assume equality holds throughout. By Corollary~\ref{cor:SV} and Theorem~\ref{thm:rigidity}, $\Omega$ agrees almost everywhere with a ball $B$. Since $\Omega$ is open, $\Omega\subset B$. We also have $T(\Omega)=T(B)$. Extend $w_\Omega$ by zero to $B$. The torsion functional has a unique maximizer in $W^{1,2}_0(B)$, so the extended function equals $w_B$. Its quasi-continuous representative vanishes quasi-everywhere on $B\backslash\Omega$, whereas $w_B>0$ in $B$. Therefore $B\backslash\Omega$ has zero  capacity. Conversely, removing a set of zero capacity from a ball changes neither $W^{1,2}_0$, the first Dirichlet eigenvalue, nor the field energy. This proves the equality statement.
\end{proof}

\begin{remark}\label{rem:range}
For $N\geq3$, Theorem~\ref{thm:main} and the examples below describe the range of the interior fraction $\Ein(\Omega)/\E(\Omega)$: its supremum $\frac{N-2}{2N}$ is attained exactly by balls up to sets of measure zero, while the infimum is $0$ and is not attained.

For $0<h<1$, consider the thin shell $\Omega_h=B_1\backslash B_{1-h}$. Newton's theorem gives
\eq{
    \abs{\nabla u_{\Omega_h}(x)}
    =\frac{\abs{\Omega_h\cap B_{\abs{x}}}}{\omega_{N-1}\abs{x}^{N-1}}
    \leq h
    \qquad (x\in\Omega_h).
}
Hence
\eq{
    \Ein(\Omega_h)\leq h^2\abs{\Omega_h}=O(h^3),
}
whereas
\eq{
    \E(\Omega_h)
    \geq\frac{\abs{\Omega_h}^2}{(N-2)\omega_{N-1}\operatorname{cap}(\overline{B_1})}
}
has a right-hand side of order $h^2$.
Here
\eq{
    \operatorname{cap}(\overline{B_1})^{-1}=\min_\mu\iint\abs{x-y}^{2-N}\,\rd\mu(x)\,\rd\mu(y),
}
with the minimum taken over probability measures $\mu$ on $\overline{B_1}$. The bound holds because $\abs{\Omega_h}^{-1}\chi_{\Omega_h}\,\rd x$ is such a measure. Therefore $\Ein(\Omega_h)/\E(\Omega_h)=O(h)$ as $h\to0$.

Moreover, the infimum is not attained. Indeed, local elliptic regularity gives $u_\Omega\in W^{2,2}_{\mathrm{loc}}(\R^N)$. If $\Ein(\Omega)=0$, then $\nabla u_\Omega=0$ a.e.\ on $\Omega$, and locality of weak derivatives gives $D^2u_\Omega=0$ a.e.\ there. This contradicts $\Delta u_\Omega=1$ a.e.\ on $\Omega$.

In $\R^3$, the formulas in Example~\ref{ex:ellipsoids} give further examples. For a prolate spheroid with semiaxes $(\ell,a,a)$, $\Ein/\E$ is of order $1/\log(\ell/a)$ as $a/\ell\to0$; for an oblate spheroid with semiaxes $(R,R,h)$, it is of order $h/R$ as $h/R\to0$. For $N=2$, ellipses with aspect ratio tending to zero show that $\Ein(\Omega)/\abs{\Omega}^2$ has infimum $0$; see Example~\ref{ex:ellipsoids}.
\end{remark}

\begin{proof}[Proof of Corollary~\ref{cor:bergman}]
Identify $\C$ with $\R^2$ and write $\p=\p_z$, $\bar\p=\p_{\bar z}$, so that $\Delta=4\p\bar\p$ and $\abs{\p g}=\tfrac{1}{2}\abs{\nabla g}$ for real-valued $g$.

Let $w=w_\Omega$ and set $g=-4\p w\in L^2(\Omega)$. Then $\bar\p g=-\Delta w=1=\bar\p\bar z$ in $\mathcal{D}'(\Omega)$, so $\bar z-g$ has a holomorphic representative in $\Omega$. Since $\Omega$ is bounded, we have $\bar z-g\in A^2(\Omega)$.

If $F\in A^2(\Omega)$ and $\varphi\in C_c^\infty(\Omega)$, then
$\int_\Omega\p\varphi\,\overline F=-\int_\Omega\varphi\,\overline{\bar\p F}=0$.
By density the same holds with $\varphi=w$, hence $g\perp A^2(\Omega)$. Thus $\bar z-g$ is the orthogonal projection of $\bar z$ onto $A^2(\Omega)$, and
\eq{
    \lambda_{A^2}(\Omega)^2
    =\norm{g}_{L^2(\Omega)}^2
    =16\int_\Omega\abs{\p w}^2
    =4\int_\Omega\abs{\nabla w}^2
    =4\,T(\Omega).
}
By Corollary~\ref{cor:SV}, $4T(\Omega)\leq4\Ein(\Omega)=\int_\Omega\abs{\Cc_\Omega}^2\leq\abs{\Omega}^2/(2\pi)$. If $\lambda_{A^2}(\Omega)=\abs{\Omega}/\sqrt{2\pi}$, then equality holds in Theorem~\ref{thm:main}, and Theorem~\ref{thm:rigidity} implies that $\Omega$ is a disc up to a set of measure zero.
\end{proof}

One can also see $\lambda_{A^2}(\Omega)\leq\norm{\Cc_\Omega}_{L^2(\Omega)}$ directly: since $\bar\p\Cc_\Omega=\chi_\Omega$, the function $\bar z-\Cc_\Omega$ is holomorphic in $\Omega$, hence belongs to $A^2(\Omega)$.

\begin{proof}[Proof of Corollary~\ref{cor:bergmanND}]
Let $w=w_\Omega\in W^{1,2}_0(\Omega)$ solve $-\Delta w=1$ and set
\eq{
    F_0\coloneqq x+N\nabla w.
}
Then $\rdiv F_0=N+N\Delta w=0$ and $\p_i(F_0)_j=\p_j(F_0)_i$ in the distributional sense. Hence $F_0\in\mathcal{A}^2(\Omega)$. For any $F\in\mathcal{A}^2(\Omega)$, integration by parts and approximation of $w$ by compactly supported smooth functions give
\eq{
    \int_\Omega \<\nabla w,F-F_0\>\,\rd x=0.
}
It follows that
\eq{
    \norm{x-F}_{L^2(\Omega)}^2
    =N^2\norm{\nabla w}_{L^2(\Omega)}^2+\norm{F-F_0}_{L^2(\Omega)}^2.
}
Taking the infimum over $F$ gives
\eq{\label{eq:bergmanND-torsion}
    \lambda_{\mathcal{A}^2}(\Omega)^2=N^2T(\Omega).
}
Moreover, $x-N\nabla u_\Omega\in\mathcal{A}^2(\Omega)$, so
\eq{
    \lambda_{\mathcal{A}^2}(\Omega)^2
    \leq N^2\int_\Omega\abs{\nabla u_\Omega}^2\,\rd x
    \leq\frac{N}{N+2}\Big(\frac{1}{s_N}\Big)^{\frac{2}{N}}\abs{\Omega}^{1+\frac{2}{N}},
}
where we used Corollary~\ref{cor:ball} in the last inequality. Equality in the final bound implies equality in Corollary~\ref{cor:ball}, hence $\Omega$ is a ball up to a set of measure zero. It is clear that balls attain equality. A bounded domain with Lipschitz boundary that agrees almost everywhere with a ball is that ball.
\end{proof}

\begin{remark}\label{rem:geometric}
Assume $\Omega$ is bounded. For each $c\in\R^N$, symmetrizing in $x$ and $y$ gives the virial identity
\eq{\label{eq:virial}
    \int_\Omega\<x-c,\nabla u_\Omega(x)\>\,\rd x
    =\frac{1}{\omega_{N-1}}\iint_{\Omega\times\Omega}\frac{\<x-c,x-y\>}{\abs{x-y}^N}\,\rd x\,\rd y
    =\frac{1}{2\omega_{N-1}}I(\Omega),
}
since the terms involving $c$ cancel by antisymmetry.
Using \eqref{eq:virial}, the inequality \eqref{eq:main} is equivalent to
\eq{
    \Big\<\nabla u_\Omega,\ \nabla u_\Omega-\frac{x-c}{N}\Big\>_{L^2(\Omega)}\leq0,
}
or, equivalently,
\eq{
    \Big\|\nabla u_\Omega-\frac{x-c}{2N}\Big\|_{L^2(\Omega)}\leq\Big\|\frac{x-c}{2N}\Big\|_{L^2(\Omega)}.
}
Thus, in $L^2(\Omega;\R^N)$ the field $\nabla u_\Omega$ lies in the closed ball having $0$ and $(x-c)/N$ as antipodal points, for every $c$. Equality holds if and only if $\Omega$ is a ball up to a set of measure zero. If its center is $c_0$, then $\nabla u_\Omega=(x-c_0)/N$ a.e.\ on $\Omega$.
\end{remark}

\begin{example}\label{ex:ellipsoids}
Let $A$ be a positive definite symmetric matrix and
$Q=\{x\in\R^N:\ \<Ax,x\><1\}$. For $x\in Q$, the numbers $\lambda(x,\theta)$ and $-\lambda(x,-\theta)$ are the roots $r$ of
\eq{
    r^2\<A\theta,\theta\>+2r\<Ax,\theta\>+\<Ax,x\>-1=0,
}
so
\eq{
    \delta(x,\theta)=-\frac{2\<Ax,\theta\>}{\<A\theta,\theta\>}.
}
By Lemma~\ref{lem:rays}, the interior field is linear: $\nabla u_Q(x)=M_Ax$, where
\eq{
    M_A=\frac{1}{\omega_{N-1}}\int_{\Sph}\frac{\theta\otimes\theta}{\<A\theta,\theta\>}\,\rd\theta\;A,
    \qquad \rtr M_A=1.
}
Thus the interior potential of a homogeneous ellipsoid is quadratic, as is classical \cite{Ke}.

For $A=\operatorname{diag}(a_1^{-2},\dots,a_N^{-2})$, the matrix $M_A=\operatorname{diag}(D_1,\dots,D_N)$ is diagonal. For $N=2$,
$D_1=\frac{a_2}{a_1+a_2}$ and $D_2=\frac{a_1}{a_1+a_2}$. For $N=3$ the $D_i$ are the classical depolarization factors
\eq{
    D_i=\frac{a_1a_2a_3}{2}\int_0^\infty\frac{\rd s}{(a_i^2+s)\sqrt{(a_1^2+s)(a_2^2+s)(a_3^2+s)}}
}
\cite{Ke}.

Since $\int_Q x_i^2\,\rd x=\abs{Q}\,a_i^2/(N+2)$, the virial identity \eqref{eq:virial} gives
\eq{
    \frac{\Ein(Q)}{\frac{1}{2N\omega_{N-1}}I(Q)}
    = N\,\frac{\sum_i D_i^2a_i^2}{\sum_i D_ia_i^2}.
}
For an ellipse with semiaxes $a,b$ this equals $\frac{4ab}{(a+b)^2}=1-\big(\frac{a-b}{a+b}\big)^2$; equivalently,
$\int_Q\abs{\Cc_Q}^2=\frac{\abs{Q}^2}{2\pi}\cdot\frac{4ab}{(a+b)^2}$.

For $N\geq3$ one obtains
\eq{
    \frac{\Ein(Q)}{\E(Q)}=\frac{N-2}{2}\,\frac{\sum_i D_i^2a_i^2}{\sum_i D_ia_i^2}.
}
By Theorem~\ref{thm:main}, $N\sum_iD_i^2a_i^2\leq\sum_iD_ia_i^2$. Therefore the depolarization factors satisfy
\eq{\label{eq:depol}
    \sum_{i=1}^N D_i\,a_i^2\,(1-ND_i)\geq 0,
}
with equality if and only if $a_1=\dots=a_N$.

The $D_i$ are homogeneous of degree zero in $(a_1,\dots,a_N)$. Formula~\eqref{eq:depol} is a weighted inequality for these factors.
Since $Q$ is convex, the entire deficit comes from the first term in \eqref{eq:deficit}, which can be computed explicitly from the formula for $\delta$. Numerical values for $N=3$ are listed in Table~\ref{tab:ellipsoids}. For the nonspherical examples, the first term of \eqref{eq:deficit}, evaluated by spherical quadrature, agrees with $\mathcal{D}(Q)$ to relative error below $10^{-10}$. The diagonal of $M_A$ agrees with the integral formula for the $D_i$ to within $10^{-11}$. These are numerical consistency checks.
\end{example}

\begin{table}[ht]
\centering
\begin{tabular}{cccc}
\toprule
Semi-axes $(a_1,a_2,a_3)$ & $(D_1,D_2,D_3)$ & $\Ein(Q)/\E(Q)$ & $\Ein(Q)/\Ein(Q^*)$\\
\midrule
$(1,1,1)$ & $(0.3333,0.3333,0.3333)$ & $0.16667$ & $1$\\
$(2,1,1)$ & $(0.1736,0.4132,0.4132)$ & $0.15190$ & $0.8731$\\
$(1,1,\tfrac{1}{2})$ & $(0.2364,0.2364,0.5272)$ & $0.14990$ & $0.8632$\\
$(3,2,1)$ & $(0.1563,0.2672,0.5765)$ & $0.13725$ & $0.7612$\\
$(1,1,\tfrac{1}{5})$ & $(0.1248,0.1248,0.7505)$ & $0.09598$ & $0.4707$\\
\bottomrule
\end{tabular}
\caption{Ellipsoids in $\R^3$. The bound in Theorem~\ref{thm:main} is $\frac{1}{6}\approx0.16667$, and $Q^*$ is the ball with $\abs{Q^*}=\abs{Q}$.}
\label{tab:ellipsoids}
\end{table}

\begin{example}\label{ex:twoballs}
If $E\subset\R$ is the union of two intervals of lengths $a$ and $b$ separated by a gap of length $g$, then a direct computation gives $\G(E)=2abg$.

Let $\Omega=B_1(0)\cup B_1(3e_1)\subset\R^3$. By Newton's theorem,
$ I(\Omega)=2I(B_1)+\tfrac{2}{3}\abs{B_1}^2$, so
$\frac{1}{6}\E(\Omega)=\frac{92\pi}{405}\approx0.71365$.

The interior field in $B_1(0)$ is
\eq{
    \nabla u_\Omega(x)=\frac{x}{3}+\frac{x-3e_1}{3\abs{x-3e_1}^3},
}
and the other ball is treated by symmetry. The cross term integrates to zero by the mean-value property. Direct integration gives
\eq{
    \Ein(\Omega)=\frac{8\pi}{45}+\frac{8\pi}{9}\left(\frac{1}{16}-\frac{\log2}{12}\right)\approx0.57174,
}
so the interior fraction is about $0.1335$. Numerically, the deficit $\mathcal{D}(\Omega)\approx0.14191$ splits into $0.0679$ from the first term in \eqref{eq:deficit} and $0.0740$ from the nonconvexity term. The latter is obtained by integrating $\G(E)=2abg$ over the lines meeting both balls.
\end{example}

\section{Comparison with the pointwise inequality}\label{sec:martin}

We compare Corollary~\ref{cor:ball} with Martin's sharp pointwise inequality.

Martin \cite{Ma} works in Clifford analysis, with the Cauchy kernel $E(x)=\bar x/(\omega_{N-1}\abs{x}^N)$ and the Teodorescu transform
\eq{
    T_\Omega f(x)=\int_\Omega E(x-y)f(y)\,\rd y.
}
Clifford conjugation preserves the Euclidean norm, so
\eq{
    \abs{T_\Omega1(x)}=\abs{\nabla u_\Omega(x)}\qquad(x\in\R^N).
}

If $\Omega$ is bounded, then $\abs{\nabla u_\Omega}$ is subharmonic on $\R^N\backslash\overline\Omega$ (as the modulus of the gradient of a harmonic function) and tends to $0$ at infinity. Hence
\eq{
    \sup_{\R^N}\abs{\nabla u_\Omega}=\max_{\overline\Omega}\abs{\nabla u_\Omega}.
}
Thus Martin's inequality controls the maximum of $\abs{\nabla u_\Omega}$, while Corollary~\ref{cor:ball} controls its $L^2(\Omega)$ norm.

In the planar case, $E(z)=\bar z/(2\pi\abs{z}^2)$ is half the Cauchy kernel $1/(\pi z)$, so $\Cc_\Omega=2T_\Omega1$ and $\abs{\Cc_\Omega}=2\abs{\nabla u_\Omega}$.

For $N\geq2$, set
\eq{
    K_N\coloneqq \{z\in\R^N:\ \abs{z}^N<z_N\}.
}
In polar coordinates $z=\rho\theta$ with $\theta_N=\cos\varphi$, the set $K_N$ is given by $\{\rho^{N-1}<\cos\varphi\}$.
Using the standard identity
\eq{
    \int_{\{\theta_N>0\}}\theta_N^{\,a}\,\rd\theta
    =\pi^{(N-1)/2}\,\frac{\Gamma\big(\frac{a+1}{2}\big)}{\Gamma\big(\frac{a+N}{2}\big)}
    \qquad(a>-1),
}
with $a=\frac{N}{N-1}$, we obtain
\eq{
    \abs{K_N}
    =\frac{1}{N}\int_{\{\theta_N>0\}}\theta_N^{\,N/(N-1)}\,\rd\theta
    =\frac{\pi^{(N-1)/2}}{N}\,\frac{\Gamma\big(\frac{2N-1}{2N-2}\big)}{\Gamma\big(\frac{2N-1}{2N-2}+\frac{N-1}{2}\big)}.
}
For $N=2$, $K_2$ is the disc of radius $\tfrac{1}{2}$ centered at $(0,\tfrac{1}{2})$; for $N=3$, $\abs{K_3}=\frac{4\pi}{15}=\tfrac{1}{5}\,s_3$.

\begin{theorem}[\cite{Ma}]\label{thm:martin}
Let $\Omega\subset\R^N$ be measurable with $0<\abs{\Omega}<\infty$. For every $x\in\R^N$,
\eq{\label{eq:martin}
    \abs{\nabla u_\Omega(x)}\leq\mu_N\abs{\Omega}^{1/N},\qquad \mu_N=\frac{N}{\omega_{N-1}}\,\abs{K_N}^{(N-1)/N},
}
with equality at $x$ if and only if $\Omega$ coincides up to a set of measure zero with $x+R(\lambda K_N)$ for some $R\in SO(N)$ and $\lambda>0$. Consequently
\eq{
    \sup_{\R^N}\abs{\nabla u_\Omega}\leq c_N\sup_{\R^N}\abs{\nabla u_{\Omega^*}},
}
where
\eq{
    c_N=N\Big(\frac{\abs{K_N}}{s_N}\Big)^{(N-1)/N},
}
and $c_N$ is sharp. One has $c_2=1$, and $c_N>1$ for $N\geq3$.
\end{theorem}

This is the inequality of \cite[Thm.~2.5 and (2.30)]{Ma}, stated in our normalization. For $N=2$ one has $\abs{K_2}=\pi/4$ and $\mu_2=1/(2\sqrt\pi)$, so \eqref{eq:martin} reduces to the Ahlfors--Beurling inequality \eqref{eq:AB}, with the extremal disc $x+R(\lambda K_2)$ having $x$ on its boundary. We include the short proof in order to fix the constants.

\begin{proof}
We may assume $\nabla u_\Omega(x)\ne0$ and, after a rotation, that $\nabla u_\Omega(x)$ is a positive multiple of $-e_N$. Then
\eq{
    \abs{\nabla u_\Omega(x)}=\frac{1}{\omega_{N-1}}\int_{\Omega-x}g(z)\,\rd z,
    \qquad g(z)=\frac{z_N}{\abs{z}^N}.
}
The function $g$ is homogeneous of degree $1-N$, and $\{g>1\}=K_N$, so $\{g>t\}=t^{-1/(N-1)}K_N$ for $t>0$. As $t$ decreases from $\infty$ to $0$, these sets increase from $\emptyset$ to $\{z_N>0\}$. Every level set of $g$ has measure zero.

Choose $t>0$ such that $\abs{\{g>t\}}=\abs{\Omega}$, i.e.\ $\{g>t\}=\lambda K_N$ with $\lambda=(\abs{\Omega}/\abs{K_N})^{1/N}$. For every measurable $A$ with $\abs{A}=\abs{\Omega}$,
\eq{
    \int_A g-\int_{\lambda K_N} g
    =\int_{A\backslash\lambda K_N}(g-t)-\int_{\lambda K_N\backslash A}(g-t)
    \leq0,
}
with equality if and only if $\abs{A\triangle\lambda K_N}=0$ (bathtub principle, \cite[Thm.~1.14]{LL}).

Since $g(\lambda z)=\lambda^{1-N}g(z)$, we have $\int_{\lambda K_N}g=\lambda\int_{K_N}g$, and in polar coordinates
\eq{
    \int_{K_N}g
    =\int_{\{\theta_N>0\}}\int_0^{\theta_N^{1/(N-1)}}\theta_N\,\rd\rho\,\rd\theta
    =\int_{\{\theta_N>0\}}\theta_N^{\,N/(N-1)}\,\rd\theta
    =N\abs{K_N}.
}
Therefore
$\abs{\nabla u_\Omega(x)}\leq\frac{N\abs{K_N}}{\omega_{N-1}}\lambda=\mu_N\abs{\Omega}^{1/N}$,
with the stated equality case; the rotation $R$ accounts for the direction of $\nabla u_\Omega(x)$.
\end{proof}

For the ball $\Omega^*$ of radius $r$ one has $\abs{\nabla u_{\Omega^*}(x)}=\abs{x}/N$ for $\abs{x}\leq r$ and $\abs{\Omega^*}/(\omega_{N-1}\abs{x}^{N-1})$ for $\abs{x}\geq r$, so $\sup\abs{\nabla u_{\Omega^*}}=r/N=\frac{1}{N}(\abs{\Omega}/s_N)^{1/N}$. Dividing, $\mu_N\abs{\Omega}^{1/N}=c_N\sup\abs{\nabla u_{\Omega^*}}$. Equality in \eqref{eq:martin} at $x=0$ for $\Omega=K_N$ shows that $c_N$ is sharp, and applying \eqref{eq:martin} to $\Omega^*$ shows $c_N\geq1$. The boundary of $K_N$ is $\rho=(\cos\varphi)^{1/(N-1)}$, while a sphere through $0$ with a center on the $e_N$-axis is $\rho=2a\cos\varphi$; these agree only for $N=2$. For $N\geq3$ the ball is therefore never an equality case of \eqref{eq:martin}, and since $\sup\abs{\nabla u_{\Omega^*}}$ is attained, $c_N>1$.

\iffalse
\begin{table}[ht]
\centering
\begin{tabular}{cccc}
\toprule
$N$ & $\abs{K_N}/s_N$ & $c_N$ & $c_N^2\,\frac{N+2}{N}$\\
\midrule
$2$ & $0.25$ & $1$ & $2$\\
$3$ & $0.2$ & $1.0260$ & $1.7544$\\
$4$ & $0.17394$ & $1.0774$ & $1.7410$\\
$5$ & $0.15686$ & $1.1360$ & $1.8068$\\
$6$ & $0.14440$ & $1.1962$ & $1.9078$\\
\bottomrule
\end{tabular}
\caption{Martin's constant $c_N$, and the loss factor incurred by integrating the pointwise inequality (Proposition~\ref{prop:compare}).}
\label{tab:martin}
\end{table}
\fi 

\begin{proposition}\label{prop:compare}
Let $N\geq2$. Then:
\begin{enumerate}[(a)]
\item Integrating \eqref{eq:martin} over $\Omega$ yields
\eq{
    \frac{1}{\abs{\Omega}}\int_\Omega \abs{\nabla u_\Omega}^2\,\rd x
    \leq c_N^2\,\frac{N+2}{N}\cdot\frac{1}{N(N+2)}\Big(\frac{\abs{\Omega}}{s_N}\Big)^{2/N}.
}
The factor $c_N^2\frac{N+2}{N}$ is strictly larger than $1$. For $N=2$ it equals $2$, as in the estimate obtained by integrating the Ahlfors--Beurling inequality in the introduction.

\item For $N=2$ both inequalities are attained by discs (the pointwise one at boundary points). For $N\geq3$ the extremal classes are disjoint: $K_N$ is not a ball, so by Theorem~\ref{thm:rigidity} it cannot be an equality case of Corollary~\ref{cor:ball}; conversely, the ball is not an equality case of Theorem~\ref{thm:martin}.
\end{enumerate}
\end{proposition}

\begin{proof}
(a)
Squaring \eqref{eq:martin} and integrating over $\Omega$ gives
\eq{
    \frac{1}{\abs{\Omega}}\int_\Omega \abs{\nabla u_\Omega}^2\,\rd x\leq \mu_N^2\abs{\Omega}^{2/N}.
}
Substituting $\mu_N^2\abs{\Omega}^{2/N}=c_N^2N^{-2}(\abs{\Omega}/s_N)^{2/N}$ gives the claim. The loss factor is strictly larger than $1$ since $c_N\geq1$.

 (b)  was already noted in the proof of Theorem~\ref{thm:martin}.
\end{proof}

%The set $K_N$ is convex, since $z\mapsto\abs{z}^N-z_N$ is convex. For $N=3$, the ray representation in Lemma~\ref{lem:rays} and axisymmetric numerical quadrature give\eq{    \frac{\Ein(K_3)}{\Ein(K_3^*)}\approx0.9866,  \qquad  \frac{\Ein(K_3)}{\E(K_3)}\approx0.1651<\tfrac{1}{6}.}On the other hand,
%\eq{   \sup\abs{\nabla u_{K_3}}=c_3\sup\abs{\nabla u_{K_3^*}}\approx1.026\,\sup\abs{\nabla u_{K_3^*}}.}
%For this example the maximum field is about $2.6\%$ larger than that of the equal-volume ball, while the interior energy is about $1.3\%$ smaller.

\begin{remark}\label{rem:Lp}
For $1\leq p<\infty$, set
\eq{
    M_p(\Omega)=\Big(\frac{1}{\abs{\Omega}}\int_\Omega \abs{\nabla u_\Omega}^p\,\rd x\Big)^{1/p},
    \qquad
    M_\infty(\Omega)=\operatorname*{ess\,sup}_{x\in\Omega}\abs{\nabla u_\Omega(x)}.
}
By Corollary~\ref{cor:ball}, the ball maximizes $M_2$ at fixed volume, uniquely up to translation and sets of measure zero. For $N=2$, the Ahlfors--Beurling inequality implies that the disc also maximizes $M_\infty$.

Since $0\in\p K_N$ and the field is continuous, its value at $0$ is approached on subsets of $K_N$ of positive measure. The equality case of Theorem~\ref{thm:martin} therefore gives
\eq{
    M_\infty(K_N)=\mu_N\abs{K_N}^{1/N}=c_NM_\infty(K_N^*).
}
For $N\geq3$, this is strictly larger than $M_\infty(K_N^*)$. Since $M_p(\Omega)\to M_\infty(\Omega)$ as $p\to\infty$, the ball fails to maximize $M_p$ for all sufficiently large  $p$. 
It is an interesting to ask for which $p$ the ball is maximizer and for which the ball is not. %For $N=3$, numerical quadrature gives $M_p(K_3)/M_p(K_3^*)\approx0.9933$ at $p=2$, about $0.9994$ at $p=100$, and values above $1$ for $p$ near $110$ and larger sampled exponents. These computations indicate a transition for this family; the argument does not determine an exact critical exponent. In particular, the set of exponents for which the ball is optimal contains $2$ and is bounded above.
\end{remark}

\section{Densities and the trilinear inequality}\label{sec:densities}

We extend the inequality to densities. Let $f$ be measurable, with $0\leq f\leq1$ and $0<\int_{\R^N}f\,\rd x<\infty$. As in Section~\ref{sec:results}, define
\eq{
    U_f(x) \coloneqq
    \begin{cases}
    \displaystyle \frac{1}{2\pi}\int_{\R^2}\log\frac{\abs{x-y}}{1+\abs{y}}\,f(y)\,\rd y, & N=2,\\[8pt]
    \displaystyle (\Gamma*f)(x), & N\geq3.
    \end{cases}
}
The potential belongs to $C^1(\R^N)$ and satisfies
\eq{
    \nabla U_f(x) = (\nabla\Gamma*f)(x)
    = \frac{1}{\omega_{N-1}}\int_{\R^N}\frac{x-y}{\abs{x-y}^N}f(y)\,\rd y,
    \qquad \Delta U_f=f
}
in the distributional sense. On the line $p+\R\theta$, set
\eq{
    f_{\theta,p}(t) \coloneqq f(p+t\theta),
    \qquad \sigma_f(\theta,p) \coloneqq \int_\R f_{\theta,p}(t)\,\rd t.
}

\begin{theorem}\label{thm:densities}
Let $N\geq2$ and let $f$ be as above. Then
\eq{
    \int_{\R^N}f\abs{\nabla U_f}^2\,\rd x
    \leq \frac{1}{2N\omega_{N-1}}\iint_{\R^N\times\R^N}
    \frac{f(x)f(y)}{\abs{x-y}^{N-2}}\,\rd x\,\rd y,
}
with equality if and only if $f=\chi_B$ a.e.\ for some ball $B$. For $N\geq3$, the right-hand side equals
\eq{
    \frac{N-2}{2N}\int_{\R^N}\abs{\nabla U_f}^2\,\rd x,
}
while for $N=2$ it equals $(\int f)^2/(8\pi)$.
\end{theorem}

\begin{proof}
The proof of Lemma~\ref{lem:pot}(a)--(c) applies to $f$, using $0\leq f\leq1$ and truncating $f$ to $f\chi_{B_k}$ in the energy identity. In particular, all the energy integrals below are finite. Set
\eq{
    \Lambda_f(x,\theta) \coloneqq \int_0^\infty f(x+r\theta)\,\rd r,
    \qquad \delta_f(x,\theta) \coloneqq \Lambda_f(x,\theta)-\Lambda_f(x,-\theta).
}
The ray representation in Lemma~\ref{lem:rays} and the projection identity \eqref{eq:pointwise} hold with $\delta_f$ in place of $\delta$.

For almost every line $p+\R\theta$, the function $h\coloneqq f_{\theta,p}$ is integrable. Write $\sigma=\sigma_f(\theta,p)$ and set
\eq{
    F(t) \coloneqq \int_{-\infty}^t h(s)\,\rd s.
}
Then $F'=h$ a.e.\ and $\delta_f(p+t\theta,\theta)=\sigma-2F(t)$. The change-of-variables formula for absolutely continuous monotone functions gives
\eq{
    \int_\R h(t)(\sigma-2F(t))^2\,\rd t
    =\int_0^\sigma(\sigma-2a)^2\,\rd a
    =\frac{\sigma^3}{3}.
}
Integrating over the lines, we obtain
\eq{
    \int_{\R^N}f(x)\norm{\delta_f(x,\cdot)}^2\,\rd x
    =\frac{1}{3}\int_{\Sph}\int_{\theta^\perp}\sigma_f(\theta,p)^3\,\rd p\,\rd\theta.
}

On the other hand, Lemma~\ref{lem:BP} gives
\eq{
    &\iint_{\R^N\times\R^N}f(x)f(y)\abs{x-y}^{2-N}\,\rd x\,\rd y\\
    &\qquad=\frac{1}{2}\int_{\Sph}\int_{\theta^\perp}
    \iint_{\R^2}h(s)h(t)\abs{s-t}\,\rd s\,\rd t\,\rd p\,\rd\theta.
}
Since $0\leq F'=h\leq1$, we have $\abs{s-t}\geq\abs{F(s)-F(t)}$. Changing variables in both factors gives
\eq{
    \iint_{\R^2}h(s)h(t)\abs{F(s)-F(t)}\,\rd s\,\rd t
    =\iint_{[0,\sigma]^2}\abs{a-b}\,\rd a\,\rd b
    =\frac{\sigma^3}{3}.
}
These identities also show that $\int f\norm{\delta_f(x,\cdot)}^2\,\rd x$ is finite. We may therefore integrate the projection identity. As in Section~\ref{sec:proof}, the deficit is the sum of
\eq{
    \frac{1}{4N\omega_{N-1}}\int_{\R^N}f(x)
    \norm{\delta_f(x,\cdot)-P_1\delta_f(x,\cdot)}^2\,\rd x
}
and the nonnegative term
\eq{
    \frac{1}{4N\omega_{N-1}}\int_{\Sph}\int_{\theta^\perp}
    \iint_{\R^2}h(s)h(t)
    \big(\abs{s-t}-\abs{F(s)-F(t)}\big)\,\rd s\,\rd t\,\rd p\,\rd\theta.
}

Both terms are nonnegative, which proves the inequality.

Suppose equality holds. The last double integral vanishes for almost every line. Tonelli's theorem gives
\eq{
    &\iint_{\R^2}h(s)h(t)
    \big(\abs{s-t}-\abs{F(s)-F(t)}\big)\,\rd s\,\rd t\\
    &\qquad=2\int_\R(1-h(r))F(r)(\sigma-F(r))\,\rd r.
}
If $\sigma>0$, then $h=1$ a.e.\ on the interval $\{0<F<\sigma\}$ and $h=F'=0$ a.e.\ outside it. Thus $h$ is the characteristic function of an interval, up to a null set. If $\sigma=0$, then $h=0$ a.e.\ Choose a direction for which this conclusion holds for almost every $p$. Fubini's theorem in the coordinates $x=p+t\theta$ gives $f\in\{0,1\}$ a.e.\ Hence $f=\chi_\Omega$ a.e.\ for a measurable set $\Omega$ of finite positive measure, and Theorem~\ref{thm:rigidity} implies that $\Omega$ is a ball up to a null set. Conversely, balls attain equality by Theorem~\ref{thm:main}.

The two expressions for the right-hand side follow from the energy identity in Lemma~\ref{lem:pot}(c) and from $\abs{x-y}^0=1$.
\end{proof}

\begin{proof}[Proof of Theorem~\ref{thm:tri}]
The case $f=0$ is immediate. For $f\ne0$, apply Theorem~\ref{thm:densities} to $f/\norm{f}_{L^\infty}$. The left-hand side is cubic in $f$ and the double integral is quadratic, which gives the stated bound.

For $N\geq3$, we have
\eq{
    \norm{f}_{\dot H^{-1}}^2
    =\int_{\R^N}\abs{\nabla(-\Delta)^{-1}f}^2\,\rd x
    =\int_{\R^N}\abs{\nabla U_f}^2\,\rd x.
}
For $N=2$, the identity $\abs{\Cc f}=2\abs{\nabla U_f}$ follows by comparing the kernels. Equality holds if and only if $f/\norm{f}_{L^\infty}=\chi_B$ a.e., as claimed.
\end{proof}

\begin{remark}
Let $N\geq3$ and let $f\in L^1\cap L^\infty(\R^N)$. The H\"older inequality gives
\eq{
    \int_{\R^N}\abs{f}\abs{\nabla U_f}^2\,\rd x
    \leq\norm{f}_{L^\infty}\norm{f}_{\dot H^{-1}}^2.
}
The constant $1$ is attained by signed densities, whereas Theorem~\ref{thm:tri} gives the smaller sharp constant $(N-2)/(2N)$ for nonnegative densities.

To see that $1$ is sharp for signed densities, let $a=2^{-1/N}$, $c>0$, and
\eq{
    f\coloneqq c\big(\chi_{\{a<\abs{x}<1\}}-\chi_{\{\abs{x}<a\}}\big).
}
Since
\eq{
    \int_{B_1}f\,\rd x=c(1-2a^N)\abs{B_1}=0,
}
the radial solution of $\Delta g=f$ satisfies
\eq{
    g'(r)=r^{1-N}\int_0^r s^{N-1}f(s)\,\rd s=0\qquad(r\geq1).
}
Choosing the additive constant so that $g=0$ outside $B_1$ gives a $C^1$ solution. It equals $U_f$ by decay at infinity and uniqueness. Since $\abs{f}=c$ a.e.\ on $\{\nabla g\ne0\}$,
\eq{
    \int_{\R^N}\abs{f}\abs{\nabla U_f}^2\,\rd x
    =c\int_{\R^N}\abs{\nabla g}^2\,\rd x
    =\norm{f}_{L^\infty}\norm{f}_{\dot H^{-1}}^2.
}
Thus all of the field energy lies in the support of $f$.
\end{remark}

\section{Applications}\label{sec:apps}

\subsection{Brownian motion}

\begin{proof}[Proof of Theorem~\ref{thm:bm}]
Let $w_\Omega\in W_0^{1,2}(\Omega)$ be the torsion function. Its probabilistic representation gives
\eq{
    \mathbb{E}_x[\tau_\Omega]=2w_\Omega(x),
    \qquad \frac{1}{2}\Delta(2w_\Omega)=-1\quad\text{in }\Omega.
}
Moreover,
\eq{
    \mathbb{E}_x[L_\Omega]=\int_\Omega G(x,y)\,\rd y,
    \qquad G(x,y)=\frac{2}{(N-2)\omega_{N-1}}\abs{x-y}^{2-N},
}
where $G$ is the occupation kernel of Brownian motion on $\R^N$. Averaging over $x\in\Omega$ and applying Lemma~\ref{lem:pot}(c), we obtain
\eq{
    \abs{\Omega}\,\mathbb{E}[\tau_\Omega]&=2T(\Omega),\\
    \abs{\Omega}\,\mathbb{E}[L_\Omega]
    &=\frac{2}{(N-2)\omega_{N-1}}\iint_{\Omega\times\Omega}\abs{x-y}^{2-N}\,\rd x\,\rd y
    =2\E(\Omega).
}
Hence the claimed inequality is equivalent to
\eq{
    T(\Omega)\leq\frac{N-2}{2N}\E(\Omega),
}
which follows from Corollary~\ref{cor:SV}.

If equality holds, then both inequalities in $T(\Omega)\leq\Ein(\Omega)\leq(N-2)\E(\Omega)/(2N)$ are equalities. By Theorem~\ref{thm:rigidity}, $\Omega$ coincides a.e.\ with a ball $B$. Since $\Omega$ is open, we have $\Omega\subset B$, and
\eq{
    T(\Omega)=\Ein(\Omega)=\Ein(B)=T(B).
}
The variational characterization of torsion in $W_0^{1,2}(B)$ now shows that the zero extension of $w_\Omega$ equals $w_B$. Since $w_B>0$ in $B$ and $w_\Omega=0$ quasi-everywhere on $B\backslash\Omega$, this set has capacity zero. Conversely, removing a set of capacity zero from a ball changes neither torsion nor occupation time, and hence preserves equality.
\end{proof}

For an ellipsoid in $\R^3$ with semi-axes $a_k$ and depolarization factors $D_k$, the torsion function and Example~\ref{ex:ellipsoids} give
\eq{
    \frac{\mathbb{E}[\tau_\Omega]}{\mathbb{E}[L_\Omega]}
    =\frac{T(\Omega)}{\E(\Omega)}
    =\Big(2\sum_k a_k^{-2}\sum_k D_ka_k^2\Big)^{-1}.
}
%For semi-axes $(2,1,1)$, $(1,1,\tfrac{1}{2})$, $(3,2,1)$ and $(1,1,\tfrac{1}{5})$, the values are approximately $0.1461$, $0.1378$, $0.1204$ and $0.0662$, respectively. The value for a ball is $1/6$.

\begin{remark}\label{rem:harmonic}
Let $N\geq3$ and set $g(x)=\mathbb{E}_x[L_\Omega]=-2u_\Omega(x)$. Let $\omega_x$ denote harmonic measure for $\Omega$ viewed from $x$. The strong Markov property gives
\eq{
    \mathbb{E}_x[L_\Omega-\tau_\Omega]
    =\int_{\p\Omega}g\,\rd\omega_x.
}
The right-hand side is the probabilistic harmonic extension of $g|_{\p\Omega}$, namely $-2(u_\Omega+w_\Omega)$. Thus Theorem~\ref{thm:bm} is equivalent to
\eq{
    \int_\Omega\int_{\p\Omega}g\,\rd\omega_x\,\rd x
    \geq\frac{N+2}{2N}\int_\Omega g\,\rd x,
}
with the same equality case. Thus the expected occupation time after the first exit, averaged over the starting point in $\Omega$, is at least $(N+2)/(2N)$ of the total expected occupation time.
\end{remark}

For $N=2$, Brownian motion is recurrent and $\mathbb{E}[L_\Omega]=\infty$. In this case, the inequality $T(\Omega)\leq\abs{\Omega}^2/(8\pi)$ is the Saint-Venant inequality for $\mathbb{E}[\tau_\Omega]$.

\subsection{Aggregation patches}

Let $\rho_0>0$, let $0<\gamma<1$, and let $\Omega_0$ be a nonempty bounded open set with $C^{1,\gamma}$ boundary. We consider the patch solution $\rho(t)\chi_{\Omega_t}$ of \eqref{eq:agg} with initial density $\rho_0\chi_{\Omega_0}$. The boundary regularity persists for $0\leq t<1/\rho_0$ under aggregation and for every finite time under spreading; see \cite{BGLV}. The identities below hold throughout these intervals.

All densities in the energy computations have compact support, uniformly on compact time intervals. We use the unnormalized potential $U_\rho=\Gamma*\rho$, also when $N=2$. It has the same gradient as the potential in Section~\ref{sec:densities}. This choice of additive constant fixes the interaction energy.

For a bounded planar set, define
\eq{
    L(\Omega)\coloneqq\frac{1}{2\pi}
    \iint_{\Omega\times\Omega}\log\frac{1}{\abs{x-y}}\,\rd x\,\rd y.
}
Choose $R$ larger than the diameters of $\Omega$ and $\Omega^*$, and set
\eq{
    g(z)\coloneqq
    \begin{cases}
    \displaystyle\log\frac{2R}{\abs{z}}, & \abs{z}\leq R,\\[6pt]
    \displaystyle(\log2)\frac{R}{\abs{z}}, & \abs{z}>R.
    \end{cases}
}
The function $g$ is positive, strictly symmetric decreasing, and tends to zero at infinity. Since
\eq{
    \iint_{\Omega\times\Omega}g(x-y)\,\rd x\,\rd y
    =\abs{\Omega}^2\log(2R)+2\pi L(\Omega),
}
and the same identity holds for $\Omega^*$, Theorem~\ref{thm:RS} gives $L(\Omega)\leq L(\Omega^*)$. Equality holds only for discs, up to a null set. A direct calculation for a disc $D$ of radius $r$ gives
\eq{
    L(D)=\frac{\abs{D}^2}{2\pi}\Big(\log\frac{1}{r}+\frac{1}{4}\Big).
}

\begin{theorem}\label{thm:agg}
Let $\rho(t)\chi_{\Omega_t}$ be the patch solution above, with $\rho_0>0$, and let $\mathcal{D}$ be the deficit in \eqref{eq:main}.
\begin{enumerate}[(a)]
\item If $N\geq3$, define
\eq{
    R(t)\coloneqq\frac{\E(\Omega_t)}{\E(\Omega_t^*)}\in(0,1].
}
Then
\eq{
    \frac{\rd}{\rd t}\log R(t)
    =\mp2\rho(t)\frac{\mathcal{D}(\Omega_t)}{\E(\Omega_t)},
}
with $-$ for aggregation and $+$ for spreading.
\item If $N=2$, define
\eq{
    J(t)\coloneqq\frac{L(\Omega_t^*)-L(\Omega_t)}{\abs{\Omega_t}^2}\geq0.
}
Then
\eq{
    \frac{\rd}{\rd t}J(t)
    =\pm2\rho(t)\frac{\mathcal{D}(\Omega_t)}{\abs{\Omega_t}^2},
}
with $+$ for aggregation and $-$ for spreading.
\end{enumerate}
Under spreading, $R$ increases and $J$ decreases; under aggregation, $R$ decreases and $J$ increases. Their derivatives vanish exactly when $\Omega_t$ is a ball. Both quantities are scale invariant, with $R=1$ and $J=0$ exactly for balls.
\end{theorem}

\begin{proof}
Along particle trajectories,
\eq{
    \frac{\rD\rho}{\rD t}=-\rho\rdiv v=\pm\rho^2,
}
where the upper sign corresponds to aggregation. Hence the patch keeps uniform density
\eq{
    \rho(t)=\frac{\rho_0}{1\mp\rho_0t},
}
and $M\coloneqq\rho(t)\abs{\Omega_t}$ is conserved. Set
\eq{
    W(t)\coloneqq-\int_{\R^N}U_\rho\rho\,\rd x,
    \qquad U_\rho\coloneqq\Gamma*\rho.
}
Then $W=\rho^2\E(\Omega_t)$ for $N\geq3$ and $W=\rho^2L(\Omega_t)$ for $N=2$. By symmetry of the kernel and integration by parts,
\eq{
    \frac{\rd W}{\rd t}
    &=-2\int_{\R^N}U_\rho\p_t\rho\,\rd x
    =-2\int_{\R^N}\rho\<v,\nabla U_\rho\>\,\rd x\\
    &=\pm2\int_{\R^N}\rho\abs{\nabla U_\rho}^2\,\rd x
    =\pm2\rho^3\Ein(\Omega_t).
}

For (a), we have
\eq{
    \frac{\rd}{\rd t}\log W
    =\pm2\rho\frac{\Ein(\Omega_t)}{\E(\Omega_t)}.
}
For the comparison ball, $\E(\Omega_t^*)=\kappa_N\abs{\Omega_t}^{(N+2)/N}$, where $\kappa_N$ depends only on $N$. Since $\abs{\Omega_t}=M/\rho$,
\eq{
    \frac{\rd}{\rd t}\log\big(\rho^2\E(\Omega_t^*)\big)
    =\pm\Big(2-\frac{N+2}{N}\Big)\rho
    =\pm\frac{N-2}{N}\rho.
}
Using $R=W/(\rho^2\E(\Omega_t^*))$ and $\mathcal{D}=(N-2)\E/(2N)-\Ein$, subtraction gives
\eq{
    \frac{\rd}{\rd t}\log R
    =\mp\frac{2\rho\mathcal{D}(\Omega_t)}{\E(\Omega_t)}.
}
Moreover, $R\leq1$ by Theorem~\ref{thm:RS}.

For (b), set $W^*\coloneqq\rho^2L(\Omega_t^*)$. Writing $\abs{\Omega_t}=\pi r_t^2=M/\rho$, we have
\eq{
    W^*=\frac{M^2}{2\pi}\Big(\frac{1}{2}\log\frac{\pi\rho}{M}+\frac{1}{4}\Big).
}
Consequently,
\eq{
    \frac{\rd W^*}{\rd t}
    =\frac{M^2}{4\pi}\frac{\rho'}{\rho}
    =\pm\frac{M^2}{4\pi}\rho
    =\pm2\rho^3\frac{\abs{\Omega_t}^2}{8\pi}.
}
Thus
\eq{
    \frac{\rd}{\rd t}(W^*-W)=\pm2\rho^3\mathcal{D}(\Omega_t).
}
Dividing by the conserved mass $M^2=\rho^2\abs{\Omega_t}^2$ proves the formula for $J=(W^*-W)/M^2$.

The monotonicities and their strictness follow from Theorems~\ref{thm:main} and~\ref{thm:rigidity}. The scaling laws give scale invariance. The equality case of Theorem~\ref{thm:RS} gives $R=1$ or $J=0$ exactly for balls.
\end{proof}

\begin{corollary}\label{cor:aggint}
Let $\rho(t)\chi_{\Omega_t}$ be a spreading patch solution as in Theorem~\ref{thm:agg}, with $\rho_0>0$.
\begin{enumerate}[(a)]
\item If $N\geq3$, then
\eq{
    \int_0^\infty\rho(t)\frac{\mathcal{D}(\Omega_t)}{\E(\Omega_t)}\,\rd t
    \leq\frac{1}{2}\log\frac{1}{R(0)}.
}
\item If $N=2$, then
\eq{
    \int_0^\infty\rho(t)\frac{\mathcal{D}(\Omega_t)}{\abs{\Omega_t}^2}\,\rd t
    \leq\frac{1}{2}J(0).
}
\end{enumerate}
Since $\rho(t)=\rho_0/(1+\rho_0t)$ has divergent time integral, it follows that
\eq{
    \liminf_{t\to\infty}\frac{\mathcal{D}(\Omega_t)}{\E(\Omega_t)}&=0\quad(N\geq3),\\
    \liminf_{t\to\infty}\frac{\mathcal{D}(\Omega_t)}{\abs{\Omega_t}^2}&=0\quad(N=2).
}
\end{corollary}

\begin{proof}
Integrate the identities in Theorem~\ref{thm:agg} over $[0,T]$, use $R(T)\leq1$ and $J(T)\geq0$, and let $T\to\infty$. The liminf statements follow because a nonnegative function with finite $\rho(t)\,\rd t$ integral cannot be bounded away from zero when $\int_0^\infty\rho(t)\,\rd t=\infty$.
\end{proof}

For an ellipsoidal patch, the interior field is linear. The patch therefore remains ellipsoidal, and its semi-axes satisfy
\eq{
    a_k'=\mp\rho D_ka_k,
    \qquad \rho'=\pm\rho^2.
}
%Numerical integration starting from semi-axes $(2,1,0.7)$ in $\R^3$ with $\rho_0=1$ gives $R$ increasing from $0.928$ to $0.994$ by $t=30$ for a spreading patch. For aggregation, the same computation gives $R=0.740$ at $t=0.9$ and suggests $R\to0$ as $t\to1=1/\rho_0$. The proposed limit is numerical; it is consistent with the collapse scenarios studied in \cite{BLLe}.

\begin{remark}\label{rem:euler}
The area and logarithmic energy $L(\Omega_t)$ are conserved by the two-dimensional Euler equation for vortex patches. Hence $J$ in Theorem~\ref{thm:agg}(b) is constant along the Euler flow and monotone along \eqref{eq:agg}. The Euler flow is Hamiltonian. For the interaction energy $W=-\int U_\rho\rho$, spreading is Wasserstein gradient descent and aggregation is the reverse flow, after a constant rescaling of time. Theorem~\ref{thm:agg}(b) expresses the change of $J$ along these gradient dynamics in terms of the deficit in \eqref{eq:main}.
\end{remark}

\subsection{Energy bounds for bounded solutions}

The energy identity also holds for the bounded solutions constructed in \cite{BLLe}. We give the justification below, since the density need not be smooth.

\begin{corollary}\label{cor:energybounds}
Let $\rho\geq0$ be a solution of \eqref{eq:agg} transported by its velocity flow, with nonzero compactly supported initial data $\rho_0\in L^\infty(\R^N)$. Assume that $\rho$ is continuous in $L^1$, is locally bounded in $L^\infty$, and has support uniformly bounded on compact time intervals in its lifespan. Set
\eq{
    m_0\coloneqq\norm{\rho_0}_\infty,
    \qquad M\coloneqq\int_{\R^N}\rho\,\rd x,
    \qquad W(t)\coloneqq-\int_{\R^N}U_{\rho(t)}\rho(t)\,\rd x,
}
where $U_\rho=\Gamma*\rho$ is the unnormalized potential. The following bounds hold at every indicated time in the lifespan of the solution. For the solutions of \cite{BLLe}, this includes all $t\geq0$ under spreading and all $0\leq t<1/m_0$ under aggregation.
\begin{enumerate}[(a)]
\item If $N\geq3$, then for spreading solutions and $t\geq0$,
\eq{
    W(t)\geq W(0)(1+m_0t)^{-\frac{N-2}{N}},
}
while for aggregating solutions and $0\leq t<1/m_0$,
\eq{
    W(t)\leq W(0)(1-m_0t)^{-\frac{N-2}{N}}.
}
\item If $N=2$, then for spreading solutions and $t\geq0$,
\eq{
    W(t)\geq W(0)-\frac{M^2}{4\pi}\log(1+m_0t),
}
while for aggregating solutions and $0\leq t<1/m_0$,
\eq{
    W(t)\leq W(0)+\frac{M^2}{4\pi}\log\frac{1}{1-m_0t}.
}
\end{enumerate}
All four bounds are attained by a uniform ball patch of density $m_0$.
\end{corollary}

\begin{proof}
The characteristic formula, valid for the transported solutions under consideration (see \cite{BLLe}), gives
\eq{
    \rho(t,X_t(a))=\frac{\rho_0(a)}{1\mp t\rho_0(a)},
    \qquad
    \norm{\rho(t)}_\infty\leq\frac{m_0}{1\mp m_0t},
}
with the upper sign for aggregation.

To justify the energy identity, fix a compact time interval before blow-up. On this interval, $\rho\in C(L^1)\cap L^\infty$, its support stays in a fixed ball, and $v$ is bounded. Let $\eta_\varepsilon$ be a smooth radial mollifier supported in $B_\varepsilon$, set $\Gamma_\varepsilon=\Gamma*\eta_\varepsilon$, and define
\eq{
    W_\varepsilon=-\iint\Gamma_\varepsilon(x-y)\rho(x)\rho(y)\,\rd x\,\rd y.
}
Testing the continuity equation in both variables, with cutoffs equal to $1$ on the common support, gives in the sense of distributions in time
\eq{
    \frac{\rd W_\varepsilon}{\rd t}
    =-2\int\rho\<v,\nabla\Gamma_\varepsilon*\rho\>\,\rd x.
}
By radial symmetry, $\Gamma_\varepsilon=\Gamma$ outside $B_\varepsilon$. Direct integration on $B_\varepsilon$ gives
\eq{
    \abs{W_\varepsilon-W}\leq CM\norm{\rho}_\infty\varepsilon^2,
    \qquad
    \norm{\nabla\Gamma_\varepsilon*\rho-\nabla U_\rho}_\infty
    \leq C\norm{\rho}_\infty\varepsilon,
}
with constants uniform on the chosen time interval. Passing to the limit shows that $W$ is locally absolutely continuous and
\eq{
    \frac{\rd W}{\rd t}
    =\pm2\int_{\R^N}\rho\abs{\nabla U_\rho}^2\,\rd x
}
for almost every time.
For $N\geq3$, Theorem~\ref{thm:tri} gives
\eq{
    \Abs{\frac{\rd}{\rd t}\log W}
    \leq\frac{N-2}{N}\norm{\rho(t)}_\infty.
}
For $N=2$, it gives
\eq{
    \Abs{\frac{\rd W}{\rd t}}
    \leq\frac{M^2}{4\pi}\norm{\rho(t)}_\infty.
}
Integration in time proves the bounds. A uniform ball patch attains equality at every step. The zero solution satisfies the same bounds with $m_0=M=W=0$.
\end{proof}

The same trilinear integral occurs in the parabolic--elliptic Patlak--Keller--Segel system
\eq{
    \p_t\rho=\Delta\rho-\rdiv(\rho\nabla c),\qquad c=-U_\rho,\qquad-\Delta c=\rho.
}
For sufficiently regular decaying solutions, the interaction energy satisfies
\eq{
    \frac{\rd}{\rd t}\Big(-\frac{1}{2}\int\rho c\,\rd x\Big)
    =\int\rho^2\,\rd x-\int\rho\abs{\nabla c}^2\,\rd x.
}
Thus the trilinear integral is the contribution of chemotactic drift to interaction-energy dissipation; diffusion contributes the positive term $\int\rho^2$. For $N=2$, Theorem~\ref{thm:tri} gives
\eq{
    \int\rho\abs{\nabla c}^2\,\rd x
    \leq\frac{\norm{\rho}_{L^\infty}M^2}{8\pi},
}
with equality for a nonzero density if and only if it is a uniform disc.

\bigskip
\textbf{Acknowledgements. }
Parts of this manuscript benefited from assistance from Claude (Anthropic), a large language model, particularly in checking calculations and improving the exposition of some applications. The authors take full responsibility for the contents of the paper.

%\nocite{AhlforsBeurling50,ASZ,Al,AKL,BK,BLLe,BLR,BGLV,BLL,Bo,Br,Bu,Ca,CS,Dos,FK,FL,GGS20,GK,GT,Ka,Ga,Ke,KPS,KJ,Kl,Li,LL,Ma,OR,Polya48,PS,Ri,San,WP,So}
\printbibliography

@article{AhlforsBeurling50,
  author  = {Ahlfors, L. V. and Beurling, A.},
  title   = {Conformal invariants and function-theoretic null-sets},
  journal = {Acta Mathematica},
  volume  = {83},
  pages   = {101--129},
  year    = {1950}
}

@article{Polya48,
  title={Torsional rigidity, principal frequency, electrostatic capacity and symmetrization},
  author={P{\'o}lya, G.},
  journal={Quarterly of Applied Mathematics},
  volume={6},
  number={3},
  pages={267--277},
  year={1948}
}

@article{GGS20,
  title   = {Isoperimetric inequalities for {B}ergman analytic content},
  author  = {Gardiner, S. J. M. and Ghergu, M. and Sj{\"o}din, T.},
  journal = {Indiana Univ. Math. J.},
  volume  = {69},
  number  = {4},
  pages   = {1231--1249},
  year    = {2020}
}

@article{ASZ,
  author = {D. Aharonov and M. M. Schiffer and L. Zalcman},
  title = {{Potato kugel}},
  journal = {Israel J. Math.},
  volume = {40},
  year = {1981},
  pages = {331--339},
}

@article{Al,
  author = {H. Alexander},
  title = {{Projections of polynomial hulls}},
  journal = {J. Funct. Anal.},
  volume = {13},
  year = {1973},
  pages = {13--19},
}

@article{AKL,
  author = {J. M. Anderson and D. Khavinson and V. Lomonosov},
  title = {{Spectral properties of some integral operators arising in potential theory}},
  journal = {Quart. J. Math.},
  volume = {43},
  year = {1992},
  pages = {387--407},
}

@article{BK,
  author = {C. B\'en\'eteau and D. Khavinson},
  title = {{The isoperimetric inequality via approximation theory and free boundary problems}},
  journal = {Comput. Methods Funct. Theory},
  volume = {6},
  year = {2006},
  pages = {253--274},
}

@article{BLLe,
  author = {A. L. Bertozzi and T. Laurent and F. L\'eger},
  title = {{Aggregation and spreading via the Newtonian potential: the dynamics of patch solutions}},
  journal = {Math. Models Methods Appl. Sci.},
  volume = {22},
  year = {2012},
  pages = {1140005},
  note = {suppl.~1},
}

@article{BGLV,
  author = {A. L. Bertozzi and J. Garnett and T. Laurent and J. Verdera},
  title = {{The regularity of the boundary of a multidimensional aggregation patch}},
  journal = {SIAM J. Math. Anal.},
  volume = {48},
  year = {2016},
  pages = {3789--3819},
}

@article{BLL,
  author = {H. J. Brascamp and E. H. Lieb and J. M. Luttinger},
  title = {{A general rearrangement inequality for multiple integrals}},
  journal = {J. Funct. Anal.},
  volume = {17},
  year = {1974},
  pages = {227--237},
}

@article{Bo,
  author = {C. Borell},
  title = {{Greenian potentials and concavity}},
  journal = {Math. Ann.},
  volume = {272},
  year = {1985},
  pages = {155--160},
}

@article{Br,
  author = {L. Brasco},
  title = {{On torsional rigidity and principal frequencies: an invitation to the Kohler--Jobin rearrangement technique}},
  journal = {ESAIM Control Optim. Calc. Var.},
  volume = {20},
  year = {2014},
  pages = {315--338},
}

@article{Bu,
  author = {A. Burchard},
  title = {{Cases of equality in the Riesz rearrangement inequality}},
  journal = {Ann. of Math. (2)},
  volume = {143},
  year = {1996},
  pages = {499--527},
}

@article{Ca,
  author = {X. Cabr\'e},
  title = {{Elliptic PDEs in probability and geometry: symmetry and regularity of solutions}},
  journal = {Discrete Contin. Dyn. Syst.},
  volume = {20},
  year = {2008},
  pages = {425--457},
}

@misc{CS,
  author = {M. J. Colbrook and S. Sadeghi},
  title = {{A sharp isoperimetric inequality for the Neumann--Poincar\'e operator in every dimension}},
  year = {2026},
  eprint = {2608.20466},
  eprinttype = {arxiv},
  note = {Preprint},
}

@article{Dos,
  author = {M. R. Dostani\'c},
  title = {{The properties of the Cauchy transform on a bounded domain}},
  journal = {J. Operator Theory},
  volume = {36},
  year = {1996},
  pages = {233--247},
}

@incollection{FK,
  author = {M. Fleeman and D. Khavinson},
  title = {{Approximating $\bar z$ in the Bergman space}},
  booktitle = {Recent Progress on Operator Theory and Approximation in Spaces of Analytic Functions},
  series = {Contemporary Mathematics},
  volume = {679},
  publisher = {American Mathematical Society},
  year = {2016},
  pages = {79--90},
}

@article{FL,
  author = {M. Fleeman and E. Lundberg},
  title = {{The Bergman analytic content of planar domains}},
  journal = {Comput. Methods Funct. Theory},
  volume = {17},
  year = {2017},
  pages = {369--379},
}

@article{GK,
  author = {T. W. Gamelin and D. Khavinson},
  title = {{The isoperimetric inequality and rational approximation}},
  journal = {Amer. Math. Monthly},
  volume = {96},
  year = {1989},
  pages = {18--30},
}

@book{GT,
  author = {D. Gilbarg and N. S. Trudinger},
  title = {{Elliptic Partial Differential Equations of Second Order}},
  year = {2001},
  publisher = {Springer},
  series = {Classics in Mathematics},
  note = {Reprint of the 1998 edition},
}

@article{Ka,
  author = {D. Kalaj},
  title = {{The $L^2$ norm of the interior Cauchy transform: beyond the first Dirichlet eigenvalue}},
  journal = {Potential Anal.},
  volume = {65},
  year = {2026},
  eid = {26},
}

@book{Ga,
  author = {R. J. Gardner},
  title = {{Geometric Tomography}},
  year = {2006},
  publisher = {Cambridge University Press},
  location = {Cambridge},
  series = {Encyclopedia of Mathematics and its Applications},
  volume = {58},
  edition = {2},
}

@book{Ke,
  author = {O. D. Kellogg},
  title = {{Foundations of Potential Theory}},
  year = {1929},
  publisher = {Springer},
  location = {Berlin},
}

@article{KJ,
  author = {M.-T. Kohler--Jobin},
  title = {{Une m\'ethode de comparaison isop\'erim\'etrique de fonctionnelles de domaines de la physique math\'ematique~I}},
  journal = {Z. Angew. Math. Phys.},
  volume = {29},
  year = {1978},
  pages = {757--766},
}

@article{Kl,
  author = {M. S. Klamkin},
  title = {{Optimization problems in gravitational attraction}},
  journal = {J. Math. Anal. Appl.},
  volume = {37},
  year = {1972},
  pages = {239--254},
}

@article{Li,
  author = {E. H. Lieb},
  title = {{Existence and uniqueness of the minimizing solution of Choquard's nonlinear equation}},
  journal = {Stud. Appl. Math.},
  volume = {57},
  year = {1977},
  pages = {93--105},
}

@book{LL,
  author = {E. H. Lieb and M. Loss},
  title = {{Analysis}},
  year = {2001},
  publisher = {American Mathematical Society},
  series = {Graduate Studies in Mathematics},
  volume = {14},
  edition = {2},
}

@article{Ma,
  author = {M. Martin},
  title = {{Higher-dimensional Ahlfors--Beurling type inequalities in Clifford analysis}},
  journal = {Proc. Amer. Math. Soc.},
  volume = {126},
  year = {1998},
  pages = {2863--2871},
}

@article{OR,
  author = {J.-F. Olsen and M. C. Reguera},
  title = {{On a sharp estimate for Hankel operators and Putnam's inequality}},
  journal = {Rev. Mat. Iberoam.},
  volume = {32},
  year = {2016},
  pages = {495--510},
}

@book{PS,
  author = {G. P\'olya and G. Szeg\H{o}},
  title = {{Isoperimetric Inequalities in Mathematical Physics}},
  year = {1951},
  publisher = {Princeton University Press},
  series = {Annals of Mathematics Studies},
  volume = {27},
}

@article{Ri,
  author = {F. Riesz},
  title = {{Sur une in\'egalit\'e int\'egrale}},
  journal = {J. London Math. Soc.},
  volume = {5},
  year = {1930},
  pages = {162--168},
}

@book{San,
  author = {L. A. Santal\'o},
  title = {{Integral Geometry and Geometric Probability}},
  year = {1976},
  publisher = {Addison--Wesley},
  series = {Encyclopedia of Mathematics and its Applications},
  volume = {1},
}

@article{So,
  author = {S. L. Sobolev},
  title = {{On a theorem of functional analysis}},
  journal = {Mat. Sb. (N.S.)},
  volume = {4(46)},
  year = {1938},
  pages = {471--497},
  note = {English transl. in Amer. Math. Soc. Transl. (2) \textbf{34} (1963), 39--68},
}

\end{document}